\documentclass[reqno,12pt]{amsart}
\usepackage{tikz}
\usepackage[margin=1in]{geometry}
\usepackage{graphicx}
\usepackage{amsmath}
\usepackage{amssymb}
\usepackage{amsthm}
\usepackage{enumitem}
\usepackage{bbm}
\usepackage[l2tabu,orthodox]{nag}
\usepackage[all,error]{onlyamsmath}
\usepackage[strict]{csquotes}
\usepackage{url}
\usepackage{color}

\newtheorem{theorem}{Theorem}
\newtheorem{lemma}[theorem]{Lemma}
\newtheorem{corollary}[theorem]{Corollary}
\newtheorem{hypothesis}[theorem]{Hypothesis}
\theoremstyle{definition}

\newtheorem{remark}[theorem]{Remark}
\newtheorem{example}[theorem]{Example}

\numberwithin{equation}{section}
\numberwithin{theorem}{section}

\newcommand{\YL}[1]{{\bf\color{blue}(YL: #1)}}

\newcommand{\wcT}{\widehat{\mathcal T}}
\newcommand{\diag}{\operatorname{diag}}
\newcommand{\card}{\operatorname{card}}
\newcommand{\spec}{\operatorname{Sp}}
\newcommand{\dom}{\operatorname{dom}}
\newcommand{\bbC}{\mathbb{C}}
\newcommand{\cT}{\mathcal{T}}
\newcommand{\cV}{\mathcal{V}}
\newcommand{\cW}{\mathcal{W}}
\newcommand{\cE}{\mathcal{E}}
\newcommand{\wL}{\widetilde{L}}
\newcommand{\e}{{\textrm e}}
\renewcommand{\i}{{\textrm i}}
\newcommand{\be}{{\mathbf e}}
\newcommand{\bu}{{\mathbf u}}
\newcommand{\im}{\operatorname{Im}}
\newcommand{\re}{\operatorname{Re}}
\newcommand{\rank}{\operatorname{rank}}
\renewcommand{\sl}{\sqrt{\lambda}}

\allowdisplaybreaks
\begin{document}

\author{O. Boyko}\address{South Ukrainian National Pedagogical University, Odesa, Ukraine}\email{boykohelga@gmail.com}
\author{Y. Latushkin}\address{University of Missouri, Columbia, USA}\email{latushkiny@missouri.edu}
\author{V. Pivovarchik}\address{South Ukrainian National Pedagogical University, Odesa, Ukraine, and\\ University of Vaasa, Finland} \email{vpivovarchik@gmail.com} 
\title{Recovering the shape of a quantum tree by scattering data}
\date{\today}


\begin{abstract}     We consider a scattering problem generated by the Sturm-Liouville equation on a tree which consists of an equilateral compact subtree and a half-infinite lead attached to its root.   We assume that the potential on the lead is  identically zero while the potentials on the finite edges are real. 
We show how to find the shape of the  tree using the $S$-function of the scattering problem and the eigenvalues of the operators associated with the compact tree. 
\end{abstract} 



\maketitle

\renewcommand{\baselinestretch}{1.5}
\newcommand{\C}{\textbf{C}}
\newcommand{\CB}{{\mathcal B}}
\newcommand{\CH}{{\mathcal H}}
\newcommand{\CL}{{\mathcal L}}
\newcommand{\cotanh}{\mbox{cotanh}}
\newcommand{\HB}{{\mathcal HB}}
\newcommand{\N}{\textbf{N}}
\newcommand{\R}{\textbf{R}}
\newcommand{\sgn}{\mbox{sgn}}
\newcommand{\SHB}{{\mathcal SHB}}
\newcommand{\Z}{\textbf{Z}}
\newcommand{\za}{\alpha}
\newcommand{\zb}{\beta}
\newcommand{\ze}{\varepsilon}
\newcommand{\zf}{\varphi}
\newcommand{\zg}{\gamma}
\newcommand{\zk}{\kappa}
\newcommand{\zl}{\lambda}
\newcommand{\zo}{\omega}
\newcommand{\zs}{\sigma}
\newcommand{\zz}{\zeta}
\newcommand{\rf}[1]{(\ref{#1})}
\newcommand{\p}{^{\prime}}
\newcommand{\pp}{^{\prime\prime}}
%
%

\section{Introduction.}\label{sec:intro}
\setcounter{equation}{0}

In this paper, we continue the work began in \cite{BMP24, MuP},  and settle several unsolved questions for the direct Sturm-Liouville scattering problem as well as for the inverse problem of recovering the shape of a metric graph consisting of a simple connected compact equilateral tree and a lead (a half-infinite edge) attached at the root. We assume standard boundary conditions. As far as we know, the first result  on the inverse problem was obtained in \cite{GS} where the authors proved that if the lengths of the edges are non-commensurate then the $S$-function uniquely determines the shape of the graph. In general, the knowledge of the $S$-function is not sufficient to determine the
topological structure of the graph uniquely; several negative results in this direction were obtained in \cite{KurSte02}.

Distinguishing co-spectral objects by scattering data was discussed in \cite{OSh} and, for metric graphs, in \cite{BSS}. In the latter paper the authors found that scattering does not always distinguish co-spectral graphs. Paper \cite{BSS1} explained why there is a difference between the result of \cite{OSh} (saying that the scattering data resolve co-spectrality) and the result of \cite{BSS1} (saying that they do not). After these theoretical contributions experimental physicists built related microwave networks and examined problem arising there, see, for example, \cite{HLBSKS}.  

In the current paper we continue the investigation started in \cite{MuP} where it was shown that if a lead is attached to a compact connected simple equilateral graph then the $S$-function together with the eigenvalues uniquely determine the shape of the graph provided the number of vertices does not exceed $5$ while if the compact subgraph is an equilateral tree  then the statement holds true for the number of vertices that does not exceed $8$.

In the present paper we show how to determine the shape of a tree using asymptotics of the  $S$-function and eigenvalues. As in \cite{MuP} we assume that the potential on the lead is identically zero to deal with the meromorphic $S$-function. This approach originates from \cite{Re} and was used to deal with quantum graphs in \cite{P1} and \cite{LP}. In the case of meromorphic $S$-function  the corresponding Jost function  is an entire function of exponential type.

Specifically, we show how to determine a metric 
tree by given {\em scattering information}, that is, 
how to determine the shape of a metric tree $T_\infty$ obtained from a compact tree $T$  by attaching a semi-infinite lead at its root provided the following data are given: (a) The limits of a sequence of values of the scattering function $S$ along a prescribed sequence of complex numbers and (b) all common eigenvalues of two operators in $L^2(T)$ obtained, respectively,  by imposing the Dirichlet, respectively, standard boundary conditions at the root, and standard boundary conditions at all other vertices of $T$. We note that the common negative eigenvalues are isolated eigenvalues of the respective Sturm-Liouville operator in $L^2(T_\infty)$ while the positive common eigenvalues, if any, are imbedded into its essential spectrum.

In Section 2 we recall the formulation and provide proofs of some results on combinatorial trees obtained in \cite{P}.  Namely, we show  how to expand  into a branched continued fraction of a special form
 the ratio of (i) the determinant of the normalized Laplacian of a tree and (ii) the determinant of the modified  normalized Laplacian of its subtree obtained by deleting the root (an arbitrarily  chosen vertex) of the tree together with the incident edges. The coefficients in this continued fraction are the degrees of the vertices of the initial tree. Such expansions into branched continued fractions have been used earlier  in finite dimensional spectral problems on trees in \cite{Pp} and \cite{MP}.

In Section 3 we describe the direct Neumann and Dirichlet spectral problems generated by the Sturm-Liouville equation on an equilateral compact tree with, in general, nonzero potential. The Neumann problem corresponds to the operator $L_N$ equipped with the standard boundary conditions, that is, the continuity and Kirchhoff's conditions at the interior vertices and the Neumann conditions at the pendant vertices.  The Dirichlet problem corresponds to the operator $L_D$ equipped with Dirichlet condition at the root  of the tree and the standard conditions at all remaining vertices.

In Section 4 we describe a scattering problem on a non-compact tree obtained by attaching a lead (a half-infinite edge)  to a compact metric equilateral tree assuming that the potential of the lead is identically zero. 
We build machinery that allows one to describe the spectrum of the corresponding self-adjoint Sturm-Liouville operator $L_\infty$ on the non-compact tree. The essential spectrum of this operator covers the non-negative half-axis; in addition, there may exist normal eigenvalues (that is, isolated eigenvalues of finite multiplicity) as well as eigenvalues embedded into the essential spectrum. In particular, we show that the total number of the isolated eigenvalues of $L_\infty$ coincide with the total number of the isolated eigenvalues of the operator $L_N$ corresponding to the Neumann problem on the original compact part of the tree. In addition, we count the eigenvalues and their multiplicities via the roots and their multiplicities of the characteristic and the Jost functions.

In Section \ref{sec:5} we show how to find the shape of the non-compact tree with $p$ vertices and one lead attached using the $S$-function and the eigenvalues of the scattering problem on the tree described in Section \ref{sec:scattering}. This is done by recovering from the set of {\em scattering information} the determinants whose ratio is studied in Section \ref{sec:2}. The data consists of
two finite sequences of real numbers, $f_k$, $k=0,1,\ldots,p$, and $\hat{f}_{\hat{k}}$, $\hat{k}=0,1,\ldots,p-1$, obtained by taking the limits of certain infinite sequences calculated  using 
 the  values of the {\em reduced} scattering function obtained from the $S$-function that we introduce by cancelling common factors corresponding to the common eigenvalues of the operators $L_N$ and $L_D$ (either isolated or embedded into the essential spectrum of $L_\infty$). 

\textbf{Notations.}\, For a closed linear operator $L$ on a Hilbert space, we let $\dom(L)$, $\rho(L)$
and $\spec(L)$ denote its domain, resolvent set and spectrum. We
refer to \cite[Section I.2]{GK} for the definition of \textit{normal eigenvalues} (that is,
isolated eigenvalues of finite algebraic multiplicity), and denote by $\spec_{\mathrm{d}}(L)$ the
set of normal eigenvalues of $L$ and by
$\spec_{\mathrm{ess}}(L)=\spec(L)\backslash\spec_{\mathrm{d}}(L)$ the essential
spectrum.

\section{Auxiliary results}\label{sec:2}

In this section we consider combinatorial trees and forests; we refer to \cite{MP} for general introduction and terminology.
Let ${\mathcal T}$ be a  combinatorial tree with $p$ vertices $\cV=\{v_0, v^{(1)}, \dots, v^{(p-1)}\}$ rooted at $v_0$, and let $A$ be its adjacency matrix with the first row corresponding to $v_0$. Let
\begin{equation}
\label{2.1}
 D=\diag\{d(v_0), d(v^{(1)}),...,d(v^{(p-1)})\}
\end{equation} 
be the diagonal degree matrix; here and in what follows we denote by $d(v)$ the degree of the vertex $v$ in the tree $\cT$.   Let $\widehat{A}$ be the principal submatrix of $A$ obtained by deleting from $A$ the first row and the first column,  and let $\widehat{D}$ be the diagonal submatrix obtained by deleting  from $D$ the first row and the first column. We introduce notation 
\begin{equation}
\label{2.2}
\psi(z):=\det(-zD+A),\quad \widehat{\psi}(z):=\det(-z\widehat{D}+\widehat{A}), \quad z\in\bbC,
\end{equation}
and recall that the matrix $D^{-1/2}AD^{-1/2}$ is often called the {\em normalized Laplacian} of ${\mathcal T}$, and so $\psi(z)$ is its characteristic determinant.
The polynomial $\widehat{\psi}(z)$ is thus the characteristic determinant of the  \textit{modified normalized Laplacian} of a tree or a forest $\widehat{{\mathcal T}}$ as it is obtained by deleting the root together with its incident edges. The word {\em modified} is being used because  the entries of $\widehat{D}$ are still the degrees of the vertices of ${\mathcal T}$ (but {\em not} of the modified graph  ${\widehat{\mathcal T}}$). 
 To emphasize that the matrices in \eqref{2.1}, \eqref{2.2} depend on the tree $\cT$ we write, when needed, $D_\cT$ and $A_\cT$ instead of $D$ and $A$, and use notation $p_{{}_\cT}$ for the number of vertices in $\cT$.

 Throughout, we will use the following rather cumbersome but detailed notation for the vertices and subtrees of the tree $\cT$.
 The vertices adjacent to the root $v_0$ will be denoted by $v_1,\ldots,v_{d(v_0)}$. We represent $\cT$ as a union of subtrees $\cT_1$, \ldots, $\cT_{d(v_0)}$ which have a common vertex $v_0$. The subtrees of $\cT$ obtained by removing $v_0$ and incident to $v_0$ edges will be denoted by $\wcT_1,\ldots,\wcT_{d(v_0)}$ so that $v_{k_1}$ is the root of $\wcT_{k_1}$ for each $k_1=1,\ldots,d(v_0)$. For each such $k_1$ we denote by $v_{k_11},\ldots,v_{k_1(d(v_{k_1})-1)}$ the vertices of the ``second generation'', that is, the vertices adjacent to $v_{k_1}$ and different from $v_0$, and by $\wcT_{k_11},\ldots,\wcT_{k_1(d(v_{k_1})-1)}$ the ``second generation'' subtrees of $\wcT_{k_1}$ obtained by removing from $\wcT_{k_1}$ its root $v_{k_1}$ and the incident to $v_{k_1}$ edges. Thus, $v_{k_1k_2}$ is the root of the subtree $\wcT_{k_1k_2}$ of $\wcT_{k_1}$ for each $k_2=1,\ldots, d(v_{k_1})-1$. Here and in what follows we keep notation  $d(v_{k_1})$ for the degree of the vertex $v_{k_1}$ in the original tree $\cT$, {\em not} in the subtree $\wcT_{k_1}$.
 Acting inductively, we introduce the vertices $v_{k_1\ldots k_l}$ and subtrees $\wcT_{v_{k_1\ldots k_l}}$ for each $k_1,\ldots,k_l$ such that $d(v_{k_1\ldots k_{l-1}})>1$.
 Finally,  we denote by $v_{k_1\ldots k_{n-2}k_{n-1}k_n}$ with $k_n=1,\ldots,d(v_{k_1\ldots k_{n-1}})-1$ the vertices of the ``last generation'', that is, the vertices adjacent to the vertex $v_{k_1\ldots k_{n-2}k_{n-1}}$ and different from $v_{k_1\ldots k_{n-2}}$. Here, the number $n$ that depends on $k_1,k_2,\ldots,k_{n-1}$ is chosen such that $d(v_{k_1\ldots {k_n}})=1$, that is, the subtree $\wcT_{k_1\ldots k_{n-1}}$ with the root $v_{k_1\ldots k_{n-1}}$ has only one edge connecting the root with the pendant vertex 
  $v_{k_1\ldots k_{n}}=\wcT_{k_1\ldots k_{n}}$.
  
  \begin{remark}\label{rem:hats} Let $D^0_n=\diag\{1,0,\ldots, 0\}\in\bbC^{n\times n}$ for any $n\in{\mathbb N}$. For any $k=1,\ldots,d(v_0)$ the vertex $v_{k}$ has degree $d(v_{k})$ in the tree $\cT_{k}$ but the same vertex $v_{k}$ has degree $d(v_{k})-1$ in the tree $\wcT_{k}$. Thus, $\widehat{D}_{\cT_{k}}={D}_{\wcT_{k}}+D^0_{p_{{}_{T_{k}}}-1}$ and $-z\widehat{D}_{\cT_k}+\widehat{A}_{T_k}=-zD_{\wcT_k}+A_{\wcT_k}-zD^0_{p_{\wcT_k}}$ where $p_{{}_{\wcT_k}}=p_{{}_{T_{k}}}-1$ is the number of vertices in the tree $\wcT_k$.
  Decomposing the determinant of the matrix $-z\widehat{D}_{\cT_k}+\widehat{A}_{T_k}$ using its first row yields the formula
  \begin{equation}\label{detf}
  \det\big(-z\widehat{D}_{\cT_k}+\widehat{A}_{T_k}\big)=\det\big(-zD_{\wcT_k}+A_{\wcT_k}\big)-z\det\big(-z\widehat{D}_{\wcT_k}+\widehat{A}_{\wcT_k}\big)
  \end{equation}
  to be used later.
  \hfill$\Diamond$\end{remark}
  
 \begin{remark}\label{rem:enume} We let $\cE$ denote the set of edges $e\in\cE$ of the tree $\cT$ and let $g=|\cE|$ be the number of edges, $g=p-1$. We recall the well-known relation $2g=\sum_{v\in\cV}d(v)$ and enumerate the edges as follows: We let $e_1,\ldots,e_{d(v_0)}$ denote the edges connecting the root $v_0$ and the ``first generation'' vertices $v_{k_1}$, $k_1=1,\ldots,d(v_0)$. For each $k_1$ we let $e_{k_11},\ldots,e_{k_1(d(v_{k_1})-1)}$
  denote the edges connecting $v_{k_1}$ and the ``second generation'' vertices $v_{k_1k_2}$, $k_2=1,\ldots, d(v_{k_1})-1$. For each $k_1$ and $k_2$ we let  
  $e_{k_1k_21},\ldots,e_{k_1k_2(d(v_{k_1k_2})-1)}$  denote the edges connecting $v_{k_1k_2}$ and the ``third generation'' vertices $v_{k_1k_2k_3}$, $k_3=1,\ldots, d(v_{k_1k_2})-1$, and so on until we reach the edges $e_{k_1\dots k_n}$ connecting $v_{k_1\dots k_{n-1}}$ with the pendant vertex $v_{k_1\dots k_n}$ such that $d(v_{k_1\dots k_n})=1$.
  \hfill$\Diamond$ \end{remark}
 
The following theorem was proved in \cite[Theorem 3.1]{P}. For reader's convenience, we give below a proof containing more details than that in \cite{P}.

\begin{theorem}\label{Theorem 2.1}
Let ${\mathcal T}$ be a tree, $\psi$ and $\widehat{\psi}$ be defined in \eqref{2.2}, and $d(v)$ denote the degree of a vertex $v$ in the three $\cT$. Then the ratio ${\psi(z)}/{\widehat{\psi}(z)}$ can be  expanded into a branched continuous fraction  whose fragments are as follows:  The beginning fragment 
of the expansion has the form 
\[
-d(v_0)z-
\sum_{k_1=1}^{d(v_0)} \cfrac{1}{
-d(v_{k_1})z- \sum\limits_{k_2=1}^{d(v_{k_1})-1} \cfrac{1}{-d(v_{k_1k_2})z-\cdots}}.\]
The intermediate fragments of the branched continuous fraction have the form
\[
...-\sum_{k_\ell=1}^{d(v_{k_1\ldots k_{l-1}})-1}\cfrac{1}{-d(v_{k_1\ldots k_l})z-\sum\limits_{k_{l+1}=1}^{d(v_{k_1\ldots k_l})-1}\cfrac{1}{-d(v_{k_1\ldots k_lk_{l+1}})z-\cdots}}\,.
\]
The end  fragments of the branched continued fraction correspond to the vertices  $v_{k_1,\ldots,k_{n-1}}$ connected to the pendant vertices $v_{k_1\ldots k_{n-1} k_n}$; they are of the form \[
\ldots -\sum_{k_n=1}^{d(v_{k_1\ldots k_{n-1}})-1}\cfrac{1}{-z}=\ldots -\frac{d(v_{k_1\ldots k_{n-1}})-1}{-z}.\]
\end{theorem}
\begin{proof} We abbreviate $d_0=d(v_0)$, $d_{k_1}=d(v_{k_1})$, $d_{k_1k_2}=d(v_{k_1k_2})$, and so on.
For a tree $\cT$ we denote by $\Psi=\Psi_\cT$ the $p_\cT\times p_\cT$-matrix $-zD+A$ and by $\widehat{\Psi}=\widehat{\Psi}_\cT$ the $(p_\cT-1)\times (p_\cT-1)$-matrix $-z\widehat{D}+\widehat{A}$
such that $\psi(z)=\det\Psi$ and $\widehat{\psi}(z)=\det(\widehat{\Psi})$; here $p_\cT$ is the number of vertices in $\cT$. We abbreviate $p=p_\cT$, $p_{k_1}=p_{\wcT_{k_1}}$, $p_{k_1k_2}=p_{\wcT_{k_1k_2}}, \ldots $ to denote the number of vertices so that 
$p_\cT=1+\sum_{k_1=1}^{d(v_0)}p_{k_1}$, $p_{k_1}=1+\sum_{k_2=1}^{d(v_{k_1})-1}p_{k_1k_2}$ and so on. Keeping in mind Remark \ref{rem:hats}, we introduce notations
$\Psi_{k_1}=-z\widehat{D}_{\cT_{k_1}}+\widehat{A}_{\cT_{k_1}}$, $\Psi_{k_1k_2}=-z\widehat{D}_{\cT_{k_1k_2}}+\widehat{A}_{\cT_{k_1k_2}}, \dots $ to denote the  reduced modified matrices corresponding to the subtrees $\wcT_{k_1}$, $\wcT_{k_1k_2}$, \ldots, of $\cT$.
We will enumerate the columns and rows of $\Psi$,
$\Psi_{k_1}$, and so on, in accordance with our convention of numbering the vertices such that the elements of $\Psi$ are as seen in the table below.
We will use notation $\Psi^{v|v'}$ for the $(p-1)\times (p-1)$-matrix obtained from the $p\times p$ matrix $\Psi$ by deleting the row corresponding to the vertex $v$ and the column corresponding to the vertex $v'$. For instance, $\Psi^{v_0|v_0}=\widehat{\Psi}$, $\Psi_{1}^{v_1|v_1}=\widehat{\Psi}_{1}$, and so on. Similarly, $\Psi^{v| \cdot }$ will denote the $(p-1)\times p$-matrix obtained from $\Psi$ by deleting the row corresponding to the vertex $v$, and $\Psi^{ \cdot |v'}$ will denote the $p\times (p-1)$-matrix obtained from $\Psi$ by deleting the column corresponding to the vertex $v'$.

 {\small \[\begin{array}{c|ccccccccc}
\Psi &v_0 & v_1 & v_{11} & v_{111}  \dots   v_{1\dots n_1} & v_2 & v_{21} & v_{211}  \dots  v_{2\dots n_2} & v_3 \dots \\ \hline
 v_0 & -zd_0 & 1 & 0 & 0 \dots 0 & 1 & 0 & 0  \dots  0 & 1  \dots \\
v_1& 1 & -z d_1 & 1&  0 \dots  0 &  0 & 0 & 0 \dots  0 & 0  \dots  \\
v_{11} & 0 & 1 &  -z d_{11} & 1  \dots  0 &  0 & 0 & 0 \dots  0 & 0  \dots \\
v_{111} & 0 & 0 & 1 & -z d_{111}  \dots  0 & 0 & 0 & 0 \dots  0 & 0  \dots \\
\vdots & \vdots & \vdots & \vdots & \vdots  \ddots  \vdots & \vdots & \vdots & \vdots &   \vdots   \\
v_{1\dots n_1} & 0 & 0 & 0 & 0 \dots -z d_{1\dots n_1} & 0 & 0 & 0 \dots  0 & 0  \dots\\
v_2 & 1 & 0 & 0 & 0  \dots  0 & -z d_2 & 1 & 0 \dots  0 & 0  \dots \\
v_{21} & 1 & 0 & 0 & 0   \dots  0 & 1 & -z d_{21} & 1  \dots  0 & 0  \dots \\
v_{211}&0 &  0 & 0 & 0  \dots  0 & 0 & 1 & -zd_{211} \dots  0 & 0  \dots \\
\vdots & \vdots & \vdots & \vdots &   \vdots & \vdots & \vdots & \vdots  \ddots  \vdots & \vdots  \\
v_{2\dots n_2} & 0 & 0 &  0 & 0   \dots  0 & 0 & 0 & 0  \dots  0 & 0  \dots \\
v_3 & 1 &0 & 0 & 0  \dots  0 & 0 & 0 & 0   \dots  0 & -zd_3 \dots \\
\vdots & \vdots & \vdots & \vdots & \vdots & \vdots & \vdots & \vdots  \dots \vdots &   \ddots
\end{array}.\]}
To represent the matrix $\Psi$ in a block form, we denote by $\e_p$ the $p\times 1$ column vector with the first entry being $1$ and the remaining entries being $0$ and by $\e'_p=[1,0,\ldots,0]$ the  $1\times p$ row vector obtained by transposing $\e_p$, and use similar notations $\e'_{p_{11}}=[1,0,\ldots,0]$ and $\e_{p_{11}}$ for the $1\times p_{11}$ and $(p_{11}\times 1)$ vectors, and so on. The matrices $\Psi$, $\Psi_1$, $\Psi_{k_1k_2}$, $\Psi_{k_1k_2k_3\ldots}$ and so on are given by the formulas
\[\begin{bmatrix} -zd_0 & \e'_{p_1} & \e'_{p_2} & \cdots & \e'_{p_{d_0}} \\
\e_{p_1} & \Psi_1 & 0  & \cdots &0\\
\e_{p_2} & 0 & \Psi_2  &\cdots & 0\\
\vdots & \vdots & \vdots & \ddots  & \vdots\\
\e_{p_{d_0}} & 0 & 0 &\cdots &\Psi_{d_0},\end{bmatrix},
\begin{bmatrix} -zd_1 & \e'_{p_{11}} & \e'_{p_{12}} & \cdots & \e'_{p_{1(d_1-1)}} \\
\e_{p_1} & \Psi_{11} & 0  & \cdots &0\\
\e_{p_2} & 0 & \Psi_{12}  &\cdots & 0\\
\vdots & \vdots & \vdots & \ddots  & \vdots\\
\e_{p_{1(d_1-1)}} & 0 & 0 &\cdots &\Psi_{1(d_1-1)}\end{bmatrix},\]
and so on. The matrices $\widehat{\Psi}$, $\widehat{\Psi}_{k_1}$, $\widehat{\Psi}_{k_1k_2}$, and so on, are block diagonal, that is, $\widehat{\Psi}=\Psi_1\oplus\cdots\oplus\Psi_{d_0}$, $\widehat{\Psi}_1=\Psi_{11}\oplus\cdots\oplus\Psi_{1(d_1-1)}$, and so on. 
We decompose the determinant of $\Psi$ using the first row corresponding to the vertex $v_0$,
\begin{equation}\label{PP1}\begin{split}
\psi(z)&=(-1)^{1+1}(-zd_0)\det(\Psi^{v_0|v_0})
+(-1)^{1+2}\det\Psi^{v_0|v_1}+(-1)^{1+p_1+2}\det\Psi^{v_0|v_2}\\&
+(-1)^{1+p_1+p_2+2}\det\Psi^{v_0|v_3}+\ldots+(-1)^{1+p_1+\ldots+p_{d_0-1}+2}\det\Psi^{v_0|v_{d_0}}.\end{split}\end{equation}
The matrix $\Psi^{v_0|v_1}$ is block lower triangular with the blocks $\begin{bmatrix}\e_{p_1} &\Psi_1^{\cdot |v_1}\end{bmatrix}$, $\Psi_2, \ldots$, of sizes $p_1\times p_1$, $p_2\times p_2,\ldots$ on the main diagonal. Expanding the determinant of the first block using its first column we find that the determinant of the first block is equal to $\det(\Psi_1^{v_1|v_1})$. In turn,  $\Psi_1^{v_1|v_1}$ is a block diagonal matrix with the blocks $\Psi_{1k_2}$, $k_2=1,\dots,d_1-1$. As a result,
\begin{equation}\label{PP2}
\det(\Psi^{v_0|v_1})=\prod_{k_2=1}^{d_1-1}\det(\Psi_{1k_2})\cdot\prod_{k_1=2}^{d_0}\det(\Psi_{k_1}).
\end{equation}
The matrix $\Psi^{v_0|v_2}$ is a block lower triangular matrix with the diagonal blocks of the sizes $(p_1+p_2)\times(p_1+p_2)$, $p_3\times p_3, \ldots$, $p_{d_0}\times p_{d_0}$. The diagonal blocks are
\[\Psi(T_1,T_2):=\begin{bmatrix} \e_{p_1} & \Psi_1 & 0_{p_1\times(p_2-1)}\\
\e_{p_2} & 0_{p_2\times p_1} & \Psi_2^{\cdot|v_2}\end{bmatrix},\,\Psi_3,\cdots,\Psi_{d_0},\]
where the matrix $\begin{bmatrix}\e_{p_1}&\Psi_1\end{bmatrix}$ is $p_1\times(p_1+1)$ while the matrix $\Psi_2^{\cdot|v_2}$ is $p_2\times(p_2-1)$. Using the first column of $\Psi(T_1,T_2)$, we expand
\begin{equation}\label{PP3}\begin{split}
\det\Psi(T_1,T_2)&=\det\begin{bmatrix}
\Psi_1^{v_1|\cdot} & 0_{(p_1-1)\times(p_2-1)}\\
0_{1\times p_1} & 0_{1\times (p_2-1)} \\ 0_{(p_2-1)\times p_1}& \Psi_2^{\cdot|v_2}\end{bmatrix}
\\&+(-1)^{1+p_1+1}
\det\begin{bmatrix}\Psi_1& 0_{p_1\times(p_2-1)}\\
0_{(p_2-1)\times p_1}& \Psi_2^{v_2|v_2}\end{bmatrix}.
\end{split}
\end{equation}
The first determinant in the right-hand side of \eqref{PP3} is zero as the diagonal block $\begin{bmatrix}\Psi_1^{v_1|\cdot}\\0_{1\times p_1}\end{bmatrix}$ of size $p_1\times p_1$ of the respective matrix has zero last row.
It follows that
\begin{equation}\label{PP4}
\det\Psi^{v_0|v_2}=(-1)^{2+p_1}\prod_{k_2=1}^{d_2-1}\det(\Psi_{2k_2})\cdot\prod_{k_1=1,k_1\neq2}^{d_0}\det(\Psi_{k_1}).
\end{equation}
We write $\Psi^{v_0|v_3}$ as a lower triangular block diagonal matrix with the blocks $\Psi(T_1,T_2,T_3)$ and $\Psi_{k_1}$, $k_1=4,\ldots d_0$, on the main diagonal. Expanding $\det(\Psi(T_1,T_2,T_3))$ using the first column of the matrix $\Psi(T_1,T_2,T_3)$ and noting that the first two determinants in the expansion are equal to zero as the respective block matrices of sizes $(p_1+p_2)\times(p_1+p_2)$ contain a zero row, we arrive at the formula
\begin{equation}\label{PP5}
\det\Psi^{v_0|v_3}=(-1)^{2+p_1+p_2}
\prod_{k_2=1}^{d_3-1}\det(\Psi_{3k_2})\cdot\prod_{k_1=1,k_1\neq 3}^{d_0}\det\Psi_{k_1}. \end{equation}
Similar expressions hold for $k_1=4,\ldots,d_0$. Collecting \eqref{PP2}, \eqref{PP4}, \eqref{PP5} and similar expressions in \eqref{PP1} and recalling formulas \[\widehat{\psi}(z)=\det(\Psi^{v_0|v_0})=
\prod_{k_1=1}^{d_0}\det(\Psi_{k_1}), \,\det(\Psi_1^{v_1|v_1})=\prod_{k_2=1}^{d_1-1}\det(\Psi_{1k_2}), \ldots\] yields
\begin{equation}\label{PP6}
\frac{\det(\Psi)}{\det(\Psi^{v_0|v_0})}
=-zd_0-\frac{\det(\Psi_1^{v_1|v_1})}{\det(\Psi_1)}-\cdots-\frac{\det(\Psi_{d_0}^{v_{d_0}|v_{d_0}})}{\det(\Psi_{d_0})}.
\end{equation}
Applying formula \eqref{PP6} with $\Psi$ replaced by $\Psi_{k_1}$ for each $k_1=1,\ldots,d_0$ yields
\begin{equation}\label{PP7}
\frac{\det(\Psi_{k_1})}{\det(\Psi_{k_1}^{v_{k_1}|v_{k_1}})}
=-zd_{k_1}-\frac{
\det(\Psi_{k_11}^{v_{k_11}|v_{k_11}})}{\det(\Psi_{k_11})}-\cdots-\frac{\det(\Psi_{k_1(d_{k_1}-1)}^{v_{k_1(d_{k_1}-1)}|v_{k_1(d_{k_1}-1)}})}{\det(\Psi_{k_1(d_{k_1}-1)})}.
\end{equation}
Applying formula \eqref{PP6} with $\Psi$ replaced by $\Psi_{k_1k_2}$ for each $k_1=1,\ldots,d_0$ and $k_2=1,\ldots,d_{k_1}-1$ yields a formula expressing  the ratio of ${\det(\Psi_{k_1k_2})}$ and 
${
\det(\Psi_{k_1k_2}^{v_{k_1k_2}|v_{k_1k_2}})}$ via $-zd_{k_1k_2}$ and the ratio of 
${
\det(\Psi_{k_1k_2k_3}^{v_{k_1k_2k_3}|v_{k_1k_2k_3}})}$ and ${\det(\Psi_{k_1k_2k_3})}$ with $k_3=1,\ldots,d_{k_1k_2}-1$. Inductively, we obtain the assertions required in the theorem.
\end{proof}

\begin{remark}
The brunched continued fraction as in Theorem \ref{Theorem 2.1} clearly determines the matrix $-zD+A$ for the tree and the degrees $d(v)$ of all vertices, and thus determines the shape of the tree. Indeed, the entry $d_0$ in $-zD+A$ shows that the root $v_0$ has $d_0=d(v_0)$ adjacent vertices. For each of the vertices $v_{k_1}$ the number $d_{k_1}-1=d(v_{k_1})-1$ determines the number of the ``second generation'' of adjacent vertices, and so on.
\hfill$\Diamond$\end{remark}

\begin{example}\label{ex:2.4} The brunched continued fraction $-3z-\frac2z-\frac1{-2z-\frac1z}$
corresponds to the graph pictured in Figure \ref{snowfl}(b).
\hfill$\Diamond$\end{example}

\begin{remark}\label{rem:FP}
One can re-write the representation for the ratio $\psi(z)/\widehat{\psi}(z)$ in Theorem \ref{Theorem 2.1} as follows. Fix $z\in\bbC$ and denote $R_0=\psi(z)/\widehat{\psi}(z)$. Then $R_0=-d(v_0)z-\sum_{k_1=1}^{d(v_0)}\frac{1}{R_{k_1}}$, where we introduce notations $R_{k_1}=-d(v_{k_1})z-\sum_{k_2=1}^{d(v_{k_1})-1}\frac{1}{R_{k_1k_2}}$, and $R_{k_1k_2}=-d(v_{k_1k_2})z-\sum_{k_3=1}^{d(v_{k_1k_2})-1}\frac{1}{R_{k_1k_2k_3}}$, and so on, and also  
$R_{k_1\ldots k_{n-1}}=
-d(v_{k_1\ldots k_{n-1}})z-\sum_{k_n=1}^{d(v_{k_1\ldots k_{n-1}})-1}\frac{1}{R_{k_1\ldots k_{n-1}k_n}}$ and $R_{k_1\ldots k_n}=-z$, where $d(v_{k_1\ldots k_n})=1$.
Clearly, the following asymptotic relations hold: 
\begin{align}
R_0&=-d(v_0)z+O(1/z),\,
R_{k_1}=-d(v_{k_1})z+O(1/z),\label{inf1}\\
R_{k_1k_2}&=-d(v_{k_1k_2})z+O(1/z),\ldots,
R_{k_1\ldots k_{n-1}}=-d(v_{k_1\ldots k_{n-1}})z+O(1/z)\nonumber
\end{align} when $z\to+\infty$.
 \hfill$\Diamond$\end{remark}

\begin{remark}\label{rem: Rem2.6}
We stress that $\psi$ and $\widehat{\psi}$ are  polynomials of degree $p=p_{{}_\cT}$ and $p-1$ respectively, where $p$ is the number of vertices of the tree. Of course, not every polynomial may serve as $\psi$ and $\widehat{\psi}$ for a tree; these 
are very special polynomials. They can be computed for any given tree. In Table \ref{tab:BCs} we provide formulas for the polynomials $\psi$ and $\widehat{\psi}$ and their ratio for all trees with values $p=2,3,4$; all possible trees are pictured in Figure \ref{lowp}. \hfill$\Diamond$\end{remark}

\begin{figure}
\begin{tikzpicture}[scale=1.5]
\draw[fill] (-1, 0) circle [radius=.07];
\node[above] at (-1,.1) {$v_0$};
\draw[fill] (-1, -1) circle [radius=.05];
\draw [ultra thick] (-1, 0) -- (-1, -1);
\node [below] at (-1,-1.2) {$p=2$};
\end{tikzpicture}
\begin{tikzpicture}[scale=1.5]
\draw[fill] (-.5, 0) circle [radius=.07];
\node[above] at (-.5,.1) {$v_0$};
\draw[fill] (-1, -1) circle [radius=.05];
\draw [ultra thick] (-.5, 0) -- (-1, -1);
\draw[fill] (0, -1) circle [radius=.05];
\draw [ultra thick] (-.5, 0) -- (0, -1);
\hskip.3cm
\draw[fill] (0.5, 0) circle [radius=.07];
\node[above] at (.5,.1) {$v_0$};
\draw[fill] (0.5, -1) circle [radius=.05];
\draw[fill] (0.5, -.5) circle [radius=.05];
\draw [ultra thick] (.5, 0) -- (0.5, -1);
\node [below] at (-.5,-1.2) {$p=3$(a)};
\node [below] at (.5,-1.2) {$p=3$(b)};
\end{tikzpicture}
\begin{tikzpicture}[scale=1.5]
\draw[fill] (-1, 0) circle [radius=.07];
\node[above] at (-1,.1) {$v_0$};
\draw[fill] (-1, -.33) circle [radius=.05];
\draw[fill] (-1, -.66) circle [radius=.05];
\draw[fill] (-1, -1) circle [radius=.05];
\draw [ultra thick] (-1, 0) -- (-1, -1);
\node [below] at (-1,-1.2) {$p=4$(a)};
\end{tikzpicture}
\begin{tikzpicture}[scale=1.5]
\draw[fill] (-1, 0) circle [radius=.07];
\node[above] at (-1,.1) {$v_0$};
\draw[fill] (-1, -.5) circle [radius=.05];
\draw[fill] (-1.5, -1) circle [radius=.05];
\draw[fill] (-.5, -1) circle [radius=.05];
\draw [ultra thick] (-1, 0) -- (-1, -.5);
\draw [ultra thick] (-1.5, -1) -- (-1, -.5);
\draw [ultra thick] (-.5, -1) -- (-1, -.5);
\node [below] at (-1,-1.2) {$p=4$(b)};
\end{tikzpicture}
\begin{tikzpicture}[scale=1.5]
\draw[fill] (-1, 0) circle [radius=.07];
\node[above] at (-1,.1) {$v_0$};
\draw[fill] (-1, -1) circle [radius=.05];
\draw[fill] (-1.5, -1) circle [radius=.05];
\draw[fill] (-.5, -1) circle [radius=.05];
\draw [ultra thick] (-1, 0) -- (-1, -.5);
\draw [ultra thick] (-1, 0) -- (-1, -1);
\draw [ultra thick] (-1.5, -1) -- (-1, 0);
\draw [ultra thick] (-.5, -1) -- (-1, 0);
\node [below] at (-1,-1.2) {$p=4$(c)};
\end{tikzpicture}
\begin{tikzpicture}[scale=1.5]
\draw[fill] (-.5, 0) circle [radius=.07];
\node[above] at (-.5,.1) {$v_0$};
\draw[fill] (-1, -1) circle [radius=.05];
\draw[fill] (-.33, -.33) circle [radius=.05];
\draw [ultra thick] (-.5, 0) -- (-1, -1);
\draw[fill] (0, -1) circle [radius=.05];
\draw [ultra thick] (-.5, 0) -- (0, -1);
\node [below] at (-.5,-1.2) {$p=4$(d)};
\end{tikzpicture}
\caption{All possible trees for $p=2,3,4$.}\label{lowp}
\end{figure}
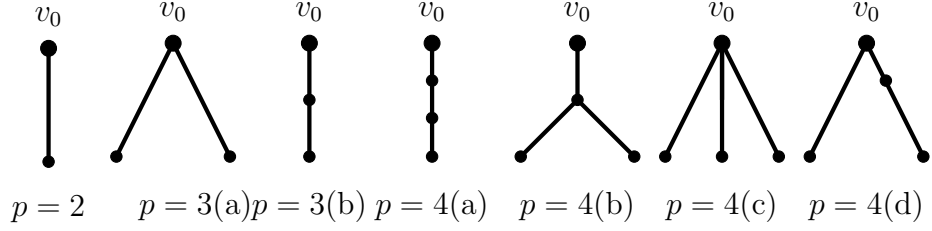

\begin{table}[t]
  \centering
  \footnotesize
  \bgroup
  \def\arraystretch{1.2}
  \begin{tabular}{|c|c|c|c|}
    \hline
    Tree
    & $\psi$
    & $\widehat{\psi}$
    & $\psi/\widehat{\psi}$ \\ \hline
    $p=2$ & $z^2-1$
    & $-z$
    & $-z+\frac1z$ \\ \hline
     $p=3 (a) $ & $2z(-z^2+1)$
    & $z^2$
    & $-2z+\frac2z$ \\ \hline
    $p=3 (b) $ & $-2z^3+2z$
    & $2z^2-1$
    & $-z-\frac{1}{-2z-\frac{1}{-z}}$ \\ \hline
    $p=4 (a) $ & $4z^4-5z^2+1$
    & $-4z^3+3z$
    & $-z-\frac{1}{-2z-\frac{1}{-2z-\frac{1}{-z}}}$ \\ \hline
    $p=4 (b) $ & $3z^2(z^2-1)$
    & $-3z^3+2z$
    & $-z+\frac1{3z-\frac2z}$ \\ \hline
    $p=4 (c) $ & $3z^2(z^2-1)$
    & $-z^3$
    & $-3z+\frac3{z}$ \\ \hline
    $p=4 (d) $ & $4z^4-5z^2+1$
    & $-2z^3+z$
    & $-2z-\frac{1}{-z}-\frac{1}{-2z-\frac{1}{-z}}$ \\ \hline
    \end{tabular}
  \egroup
  \caption{Functions $\psi$, $\widehat{\psi}$ and their ratios for low values of $p$.}
  \label{tab:BCs}
\end{table}

Theorem \ref{Theorem 2.2} allows one to use the following ``trial and error'' inductive algorithm to recover the shape of the graph provided the ratio of the functions $\psi$ and $\widehat{\psi}$ is given. That is, let us suppose that $\psi$ and $\widehat\psi$ are two given polynomial corresponding to a tree $\cT$ as indicated in \eqref{2.2}, and let $R_0(z)=\psi(z)/\widehat{\psi}(z)$ be given. 
Without loss of generality, we may suppose that the polynomials $\psi$ and $\widehat{\psi}$ are co-prime; otherwise, we will cancel the common factors in their ratio and re-denote the numerator again by  $\psi$ and denominator by $\widehat{\psi}$.
 We will use the following steps to determine the integers $d(v_0)$, $d(v_{k_1})$, $d(v_{k_1k_2})$ and so on.

{\em Step 1.} Compute $d(v_0)=\lim_{z\to+\infty}\big(\frac{{\psi}(z)}{-z\widehat{\psi}(z)}\big)$. This relation must hold by \eqref{inf1}.

{\em Step 2.} Split the polynomial $\widehat{\psi}$ into a product of $d(v_0)$ polynomials, that is, represent $\widehat{\psi}(z)=\prod_{k_1=1}^{d(v_0)}\psi_{k_1}(z)$.

We intend to use all possible ways to split $\widehat{\psi}$ into a product. At this stage the choice of $\psi_{k_1}$ is arbitrary, and might not lead to a successful further results as an unlucky chosen $\psi_{k_1}$ might not be a characteristic polynomial for any tree. If the degree of $\psi_{k_1}$ is low one may use polynomials given in Table \ref{tab:BCs}. For each $\psi_{k_1}$, $k_1=1,\ldots d(v_0)$, we treat $\psi_{k_1}$ as the determinant of the reduced modified matrix $-z\widehat{D}_{\cT_{k_1}}+\widehat{A}_{\cT_{k_1}}$ of the subtree $\cT_{k_1}$ if so chosen $\psi_{k_1}$ can serve this purpose; otherwise, we change the splitting.

{\em Step 3.} Using the method of undetermined coefficients, find a unique polynomial $\widehat{\psi}_{k_1}$ of degree $\deg\psi_{k_1}-1$ for each $k_1=1,\ldots,d(v_0)$ such that 
\begin{equation}\label{FR}
\frac{\psi(z)}{\widehat{\psi}(z)}+d(v_0)z=-\sum_{k_1=1}^{d(v_0)}\frac{\widehat{\psi}_{k_1}(z)}{\psi_{k_1}(z)},
\end{equation}
and notice that \eqref{FR} corresponds to the formula $R_0=-d(v_0)z-\sum_{k_1=1}^{d(v_0)}\frac{1}{R_{k_1}}$ in Remark \ref{rem:FP} with $R_0=\frac{\psi(z)}{\widehat{\psi}(z)}$ and $R_{k_1}=\frac{\psi_{k_1}(z)}{\widehat{\psi}_{k_1}(z)}$, where $k_1=1,\ldots,d(v_0)$.

{\em Step 4.} Find the sum of the reciprocals of $d(v_{k_1})$ by computing the limit and letting
\begin{equation}\label{FP2}
\sum_{k_1=1}^{d(v_0)}\frac{1}{d(v_{k_1})}=\lim_{z\to+\infty}\big(-z(d(v_0)z+R_0(z))\big).
\end{equation}
This relation must hold by Remark \ref{rem:FP}.

{\em Step 5.} We now view \eqref{FP2} as a Diophantine equation for the unknown natural numbers $d(v_{k_1})$, $k_1=1,\ldots,d(v_0)$, and known limits in the right-hand side of the equation. Determine $d(v_{k_1})$, $k_1=1,\ldots,d(v_0)$,
from the equation (notice that the solution might be not unique).

{\em Step 6.} Repeat Steps 1 through 5 with $\psi,\widehat{\psi}$ and $R_0=\psi(z)/\widehat{\psi}(z)$ replaced by $\psi_{k_1},\widehat{\psi}_{k_1}$ and $R_{k_1}=\psi_{k_1}(z)/\widehat{\psi}_{k_1}(z)$ for each $k_1=1,\ldots,d(v_0)$. 

Next,  we obtain polynomials  $\psi_{k_1k_2},\widehat{\psi}_{k_1k_2}$ and their ratio $R_{k_1k_2}=\psi_{k_1k_2}(z)/\widehat{\psi}_{k_1k_2}(z)$ for $k_2=1,\ldots,d(v_{k_1})-1$. We continue this inductively to find $d(v_{k_1k_2})$, and continue this process until we find $d(v_{k_1\ldots k_{n-1}})$.

\begin{example}\label{ex:2.7}
We now illustrate the algorithm described above. Given polynomials
\[\psi(z)=-108z^{11} +258z^9-202z^7+52z^5 \text{ and }
\widehat{\psi}(z)=36z^{10}-75z^8+52z^6-12z^4,\]
we compute the limit in Step 1 above to obtain $d(v_0)=3$.
One factorizes directly the polynomial $\widehat{\psi}=-z^4(-3z^2+2)^2(-4z^2+3)$, and so we have the representation
\[\frac{\psi}{\widehat{\psi}}
=-3z-\frac{33z^9-46z^7+16z^5}{z^4(-3z^2+2)^2(-4z^2+3)}.\]
Using this, we conclude that the limit in  the right hand side of \eqref{FP2} is equal to $11/12$, and so we arrive at the following
Diophantine equation, 
\begin{equation}\label{DiEq}
\frac{1}{d_1}+\frac{1}{d_2}+\frac{1}{d_3}
=\frac{11}{12},\quad d_1,d_2,d_3\in\mathbb{N},
\end{equation}
whose solutions $d_1=d(v_1)$, $d_2=d(v_2)$ and $d_3=d(v_3)$ give us the degrees of the ``first generation'' vertices connected to the root $v_0$ in the original graph $\cT$.
We claim that the only solution of  \eqref{DiEq} is $(d_1,d_2,d_3)=(3,3,4)$. To see that, with no loss of generality we denote by $d_1$ the smallest among $d_1,d_2,d_3$. Then $d_1$ obviously can not be equal to $1$. Also, $d_1$ can not be equal to $2$ because $1/d_2+1/d_3=5/12$ has only nonreal solutions. Clearly, $d_1$ can not be equal to $4$ or any larger number as then $1/d_2+1/d_3$ becomes too small to satisfy \eqref{DiEq}. This completes  the proof of the claim.

As soon as the solution to \eqref{DiEq} has been found, our next objective is to arrange the factors of $\widehat{\psi}$ into a product of three polynomials, call them $\psi_1$, $\psi_2$ and $\psi_3$, that correspond to the subtrees of $\cT$ rooted at $v_1$, $v_2$ and $v_3$. None of the factors could be of the form $z^k$ with $k=2,3,4$, as seen from the $\psi$-column in Table \ref{tab:BCs}, and if $k=1$ (compare this with the last equation in Theorem \ref{Theorem 2.1}) then one of the numbers $d_1$, $d_2$ or $d_3$ must be equal to $1$ which is not possible as then the Diophantine equation
\eqref{DiEq} has no solutions which are natural numbers.  
As a result, the function, say, $\psi_1$ must contain as a factor the term $(-3z^2+2)$, while $\psi_2$ must contain $(-3z^2+2)$ and $\psi_3$ must contain $(-4z^2+3)$. Distributing the powers of $z$ as pre-factors to $\psi_1$, $\psi_2$ and $\psi_3$ is arbitrary and its choice does not affect the result,  and so we choose 
$\psi_1(z)=\psi_2(z)=z(-3z^2+2)$ and $\psi_3(z)=z^2(-4z^2+3)$. Applying the method of undetermined coefficients we arrive at the following decomposition,
\[
\frac{\psi}{\widehat{\psi}}=
-3z-\frac{z^2}{z(-3z^2+2)}-\frac{z^2}{z(-3z^2+2)}-\frac{z^3}{z^2(-4z^2+3)}.
\]
Therefore,
\[
\frac{\psi}{\widehat{\psi}}=
-3z-\frac{1}{-3z-\frac{1}{-z}-\frac{1}{-z}}-\frac{1}{-3z-\frac{1}{-z}-\frac{1}{-z}}-\frac{1}{-4z-\frac{1}{-z}-\frac{1}{-z}-\frac{1}{-z}},
\]
and thus we recover the tree pictured in Figure \ref{snowfl}(a).
 \hfill$\Diamond$\end{example}
 
 For some type of trees the algorithm above is simpler and leads to a unique tree. We call a rooted tree  {\em $2$-snowflake} 
 when the combinatorial distance from the root to each vertex does not exceed $2$.
 An example of a $2$-snowflake is given in Figure \ref{snowfl} (a). The following result was obtained in \cite[Theorem 3.2]{P}.


\begin{figure}\centering
\begin{tikzpicture}[scale=1]
\draw[fill] (0, 0.1) circle [radius=.09];
\draw [ultra thick] (0, 0.1) -- (-1, 0.1);
\draw[fill] (-1, 0.1) circle [radius=.05];
\draw [ultra thick] (-1.2, -.5) -- (-1, 0.1);
\draw [ultra thick] (-1.2, .5) -- (-1, 0.1);
\draw[fill] (-1.2, -.5) circle [radius=.05];
\draw[fill] (-1.2, .5) circle [radius=.05];
\draw[fill] (0, -1) circle [radius=.05];
\draw [ultra thick] (0, 0) -- (0, 2);
\draw[fill] (0, 1) circle [radius=.05];
\draw [ultra thick] (0, 1) -- (-.5, 1.5);
\draw[fill] (-.5, 1.5) circle [radius=.05];
\draw [ultra thick] (0, 1) -- (.5, 1.5);
\draw[fill] (.5, 1.5) circle [radius=.05];
\draw[fill] (0, 2) circle [radius=.05];
\draw [ultra thick] (0, 0) -- (0, -2);
\draw [ultra thick] (0, -1) -- (0, 0);
\draw[fill] (0, -2) circle [radius=.05];
\draw [ultra thick] (0, -1) -- (-1, -1.5);
\draw[fill] (-1, -1.5) circle [radius=.05];
\node [right] at (0.1, 0.1) {$v_0$};
\node [below] at (0, -2) {(a)};
\end{tikzpicture} \hskip.5cm
\begin{tikzpicture}[scale=1]
\node [left] at (-.1, 2.1) {$v_0$};
\draw[fill] (0, 2) circle [radius=.09];
\draw [ultra thick] (0, 2) -- (0, 1);
\draw[fill] (0, 1) circle [radius=.05];
\draw [ultra thick] (0, 2) -- (-1, 1);
\draw[fill] (-1, 1) circle [radius=.05];
\draw [ultra thick] (0, 2) -- (1, 1);
\draw[fill] (1, 1) circle [radius=.05];
\draw [ultra thick] (1, 1) -- (1, 0);
\draw[fill] (1, 0) circle [radius=.05];
\node [below] at (0, 0) {(b)};
\end{tikzpicture}\hskip1cm
\begin{tikzpicture}[scale=1]
\draw[fill] (0, 0.1) circle [radius=.09];
\draw [->] [ultra thick] (0, 0.1) -- (1.5, 0.1);
\draw [ultra thick] (0, 0.1) -- (-1, 0.1);
\draw[fill] (-1, 0.1) circle [radius=.05];
\draw [ultra thick] (-1.2, -.5) -- (-1, 0.1);
\draw [ultra thick] (-1.2, .5) -- (-1, 0.1);
\draw[fill] (-1.2, -.5) circle [radius=.05];
\draw[fill] (-1.2, .5) circle [radius=.05];
\draw[fill] (0, -1) circle [radius=.05];
\draw [ultra thick] (0, 0) -- (0, 2);
\draw[fill] (0, 1) circle [radius=.05];
\draw [ultra thick] (0, 1) -- (-.5, 1.5);
\draw[fill] (-.5, 1.5) circle [radius=.05];
\draw [ultra thick] (0, 1) -- (.5, 1.5);
\draw[fill] (.5, 1.5) circle [radius=.05];
\draw[fill] (0, 2) circle [radius=.05];
\draw [ultra thick] (0, 0) -- (0, -2);
\draw [ultra thick] (0, -1) -- (0, 0);
\draw[fill] (0, -2) circle [radius=.05];
\draw [ultra thick] (0, -1) -- (-1, -1.5);
\draw[fill] (-1, -1.5) circle [radius=.05];
\node [right] at (0.1, 0.3) {$v_0$};
\node [below] at (0, -2) {(c)};
\end{tikzpicture}
\hskip1cm\begin{tikzpicture}[scale=1]
\node [left] at (-.1, 2.1) {$v_0$};
\draw[fill] (0, 2.1) circle [radius=.09];
\draw [->] [ultra thick] (0, 2.1) -- (1.5, 2.1);
\draw [ultra thick] (0, 2) -- (0, 1);
\draw[fill] (0, 1) circle [radius=.05];
\draw [ultra thick] (0, 1) -- (-.5, .5);
\draw[fill] (-.5, .5) circle [radius=.05];
\draw [ultra thick] (0, 1) -- (.5, .5);
\draw[fill] (.5, .5) circle [radius=.05];
\draw [ultra thick] (.5, .5) -- (1.5, .5);
\draw[fill] (1.5, .5) circle [radius=.05];
\draw [ultra thick] (.5, .5) -- (0, 0);
\draw[fill] (0, 0) circle [radius=.05];
\draw [ultra thick] (.5, .5) -- (1, 0);
\draw[fill] (1, 0) circle [radius=.05];
\draw [ultra thick] (1.5, .5) -- (2, 1);
\draw[fill] (2, 1) circle [radius=.05];
\draw [ultra thick] (1.5, .5) -- (2, 0);
\draw[fill] (2, 0) circle [radius=.05];
\node [below] at (1, 0) {(d)};
\end{tikzpicture}
\caption{(a) A $2$-snowflake, see Example \ref{ex:2.7}; 
(b) a tree with $p=5$ vertices, see Example \ref{ex:2.4};
(c)  the $2$-snowflake as in (a) but with a lead attached to the root as discussed in Section \ref{sec:scattering}; (d) a caterpillar with a lead, see Section \ref{sec:scattering}.}
\label{snowfl}
\end{figure}
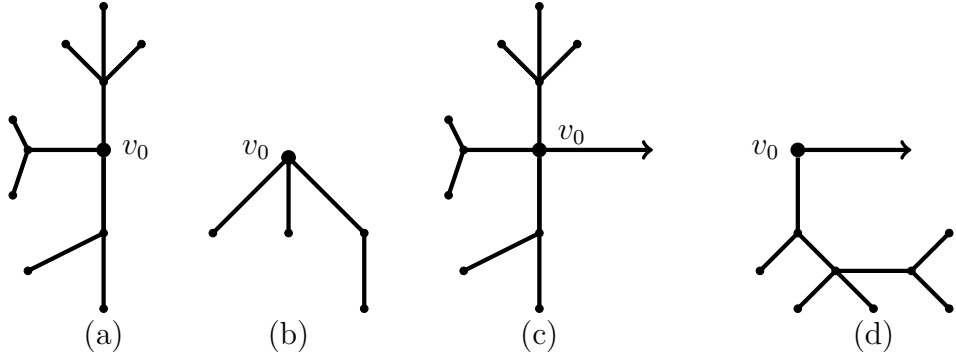

\begin{theorem}\label{Theorem 2.2}  Let ${\mathcal T}$ be a $2$-snowflake graph rooted at vertex $v_0$. The  two   functions  $\psi(z)$ and $\widehat{\psi}(z)$ defined in \eqref{2.2} uniquely determine the shape of the graph. \end{theorem}

In case of the $2$-snowflake graph we have
\[
\frac{\psi(z)}{\widehat{\psi}(z)}=
-d(v_0)z-
\sum_{k_1=1}^{d(v_0)} \frac{1}{
-d(v_{k_1})z-\frac{d(v_{k_1})-1}{-z}}
\]
where $d(v_0)$ is the degree of the the root, $d(v_{k_1})$, $k_1=1,2, ..., d(v_0)$, are the degrees of the other vertices.

\section{Quantum graph problems}\label{sec:3}

Let $T$ be an equilateral metric  tree with $p$ vertices and $g=p-1$ edges, each  of length $\ell$. We choose an arbitrary vertex $v_0$ as the root and direct all edges  away from the root  so that each vertex different from the root has exactly one incoming edge.
We consider  the Sturm-Liouville equations on the edges, 
\begin{equation}
\label{3.1}
-y_j^{\prime\prime}+q_j(x)y_j=\lambda y_j, \ \  j=1,2,..., g.
\end{equation} 
Throughout, we imposing the following
standing assumption.
\begin{hypothesis}\label{hypo:stass}
All edges of the tree $T$ have the same length $\ell$, and for all edges the potentials $q_j$ are real-valued functions from the space $L^2(0,\ell)$. \end{hypothesis}
We will now describe the spectral problem on the tree equipped with either {\em Dirichlet} or {\em standard} boundary conditions; the corresponding Sturm-Liouville operator will be denoted by $L_D$ or $L_N$, respectively.

For each edge $e_j$ incoming to a pendant vertex which is not the root 
we impose the Neumann condition 
\begin{equation}
\label{3.2}
y_j'(\ell)=0. 
\end{equation}
At  each interior vertex $v$ which is not the root we impose: (a) the continuity conditions   
\begin{equation}
\label{3.3}
y_j(\ell)-y_k(0)=0, \quad e_j\twoheadrightarrow v,\, e_k\twoheadleftarrow v,
\end{equation}
for the incoming to $v$ edge $e_j$ and  all outgoing from $v$ edges $e_k$, and (b) the Kirchhoff's conditions
\begin{equation}
\label{3.4}
y'_j(\ell)-\mathop{\sum}\limits_{k:e_k\twoheadleftarrow v} y_k'(0)=0,
\end{equation}
where the sum is taken over all  edges $e_k$ outgoing from $v$. We recall that all vertices different from the root have exactly one incoming edge.

If the root $v_0$ is an interior vertex then we denote by $e_k$, $k=1,\ldots,d(v_0)$,
all edges outgoing from the root, and at $v_0$ we impose:  (a) the continuity conditions
\begin{equation}
\label{3.5}
y_1(0)-y_k(0)=0, \, k=2,\ldots,d(v_0),
\end{equation}
for all edges $e_k$  incident to the root, and (b) the Kirchhoff's condition
\begin{equation}
\label{3.6}
\mathop{\sum}\limits_{k: e_k\sim v_0} y_k'(0)=0,
\end{equation}
where the summation is taken over all edges incident to (that is, outgoing from) the root.
If the root $v_0$ is a pendant vertex then we denote by $e_1$ the only edge outgoing from $v_0$ and impose at $v_0$ the Neumann condition
\begin{equation}
\label{3.7}
y_1'(0)=0.
\end{equation}
The above conditions (continuity plus Kirchhoff's if the vertex is not pendant or continuity plus Neumann if the vertex is pendant)  are called {\em standard}. 

Throughout, we always impose the standard conditions at all vertices except at the root. 
At the root $v_0$ one can impose one of the following two conditions: either (1)  the Dirichlet condition 
\begin{equation}
\label{3.9}
y_{k}(0)=0,\, k=1,\ldots,d(v_0),
\end{equation}
for all edges $e_{k}$ incident to the root $v_0$,  
 or (2) the standard condition; (1) is called Problem D while (2) is called  Problem N.
 Thus, the Dirichlet problem (Problem D)  consists of equations (\ref{3.1})--(\ref{3.4}) and (\ref{3.9}) while the Neumann problem consists of equations (\ref{3.1})--(\ref{3.4}) and \eqref{3.5}, \eqref{3.6} and \eqref{3.7}.

We associate with the two boundary value problems differential operators $L_D$ and $L_N$ acting in $L^2(T)=\oplus_{j=1}^gL^2(0,\ell)$ by the rule $L:(y_j)_{j=1}^g\mapsto (-y_j''+q_jy_j)_{j=1}^g$ and equipped with the domains, 
\begin{align}\label{domLD}
\dom(L_D)&=\big\{ (y_j)_{j=1}^g\in H^2(T): \text{Dirichlet at $v_0$, standard at all $v\neq v_0$}  \big\},\\ \label{domLN}
\dom(L_N)&=\big\{ (y_j)_{j=1}^g\in H^2(T): \text{standard  at $v_0$, standard at all $v\neq v_0$}  \big\},
\end{align}
where we denote by $H^2(T)=\bigoplus_{j=1}^gH^2(0,\ell)$  the Sobolev spaces.

We will now recall how to construct the
 characteristic function of the Sturm--Liouville problems described above, that is, the function of the spectral parameter whose zeros correspond to the eigenvalues of the problem or the respective operators $L_D$ and $L_N$. We refer to \cite{BerkKuch, Kostr-Shrader, Kuraso,  Pokornyj, RS12, Sch06} for a closely related to this topic material including the equality of the multiplicities of the eigenvalues and zeros of the determinants of the characteristic matrices and secular equations.

 We begin by introducing functions
 $s_j(\sqrt{\lambda},x)$ and $c_j(\sqrt{\lambda},x)$ for $x\in e_j$, $j=1,\dots,g$, as  the solutions of the Sturm-Liouville equation on the edge $e_j$ satisfying the following boundary conditions at zero, \begin{equation}\label{sc}
 s_j(\sqrt{\lambda},0)=s_j'(\sqrt{\lambda},0)-1=0 \text{ and $c_j(\sqrt{\lambda},0)-1=c_j'(\sqrt{\lambda},0)=0$}.\end{equation} 
 Here and in what follows we treat $\zeta=\sl\in\bbC$ as a complex variable such that $\zeta^2=\lambda$. We remark that the functions $c_j(\cdot,x)$ and $s_j(\cdot,x)$ are even and thus the choice of sign of $\sl$ does not matter. The functions $c_j(\cdot,x)$ and $s_j(\cdot,x)$ are entire functions of $\sl\in\bbC$ and $\lambda\in\bbC$, and we refer to Definition 12.2.2, Lemma 12.2.4 and Theorem 12.2.9 in \cite{MP0} for a discussion of this matter. In what follows we suppress the index $j$ in  the notation for $c_j$ and $s_j$ when the potentials $q_j$ are equal on all edges $e_j$, $j=1,\ldots, g$. 
 
 \begin{remark}\label{CH}
Below, we adhere to the following notational convention: The `check'-sign above a letter means that the `check'-ed object corresponds to the Sturm-Liouville problem with the potential identically equal to zero on respective edges.
In particular, $\check{s}_j(\sqrt{\lambda},x)$ is  simply
$\frac1{\sqrt{\lambda}} \sin(\sqrt{\lambda} x)$ when $\lambda\neq0$ and $\check{s}_j({0},x)=x$ while $\check{c}_j(\sqrt{\lambda},x)=\cos(\sqrt{\lambda}x)$ when $\lambda\neq0$ and $\check{c}_j({0},x)=1$. \hfill$\Diamond$\end{remark}
 
Returning to the general potentials, to construct the characteristic function, we look for coefficients $a_j, b_j$ such that the solution of \eqref{3.1} of the form 
\begin{equation}\label{yj}
y_j=a_{j}c_{j}(\sqrt{\lambda},x)+b_{j}s_{j}(\sqrt{\lambda},x), \,  x\in (0,\ell),\, j=1,\dots,g,
\end{equation}
satisfies the boundary conditions at all vertices. We substitute the expressions in \eqref{yj} with $2g$ unknowns $(a_1,\ldots,a_g,b_1,\ldots,b_g)^\top\in\bbC^{2g\times 1}$, where $\top$ is the transposition, into the boundary conditions at the vertices corresponding to the boundary value problems described above as 
Problem D and  Problem N. For each of the two problems
we obtain a system of $2g$ linear algebraic equations with unknowns $a_j, b_j$. Indeed, at each vertex $v$ the continuity condition provides $d(v)-1$ equations while the Kirchhoff/Neumann condition gives one more equation totaling to $2g=\sum_vd(v)$ equations.
For Problem D (respectively, Problem N), we denote by  $\Phi_D(\lambda)$  (respectively, by $\Phi_N(\lambda)$) the $2g\times 2g$ matrix  of this system and call it the \textit{characteristic matrix} of the respective problem. The null-space of the characteristic matrix is isomorphic to the eigenspace of the respective differential operator. We call the determinants $\phi_D(\lambda)=\det\Phi_D(\lambda)$ and $\phi_N(\lambda)=\det\Phi_N(\lambda)$ the {\em characteristic functions} of Problem D and N.
Then the equation
\[
\phi_D(\lambda):=\det\Phi_D(\lambda)=0\qquad\hbox{(respectively, $\phi_N(\lambda):=\det\Phi_N(\lambda)=0$)}
\]
fully determines the spectrum of the Sturm-Liouville problem equipped with the above mentioned vertex conditions. 

\begin{remark}\label{phinphid}
 The characteristic function $\phi_N(\lambda)$, cf.\ the tables in Appendix \ref{TT},  is expressed via the values   $c_{j}(\sqrt{\lambda},\ell)$, 
$s_{j}(\sqrt{\lambda},\ell)$, $c'_{j}(\sqrt{\lambda}, \ell)$, $s'_{j}(\sqrt{\lambda},\ell)$, which, for brevity, and slightly abusing notation, we will denote by $c_j(\ell)$, $s_j(\ell)$, $c'_j(\ell)$, $s'_j(\ell)$, respectively. It is important to notice that the characteristic matrix $\Phi_D(\lambda)$ is lower block triangular, see also the discussion after Theorem \ref{Theorem 3.2}. As a result, the characteristic function $\phi_D(\lambda)$ is expressed via the values $c_j(\ell)$, $s_j(\ell)$, $s'_j(\ell)$ (notice that, here, $c'_j(\ell)$ is {\em not} included).
\hfill$\Diamond$\end{remark}

\begin{remark}\label{rem:LawPiv} The characteristic matrices $\Phi_D$ and $\Phi_N$ discussed above correspond to the case when at all vertices but the root we impose the {\em standard} boundary conditions (Neumann if the root is pendant and continuity plus Kirchhoff if it is not). One can similarly construct the respective matrices and their determinants $\varphi_{DA}$ and $\varphi_{NA}$ assuming {\em arbitrary} self-adjoint boundary conditions at all vertices but the root (in this notation, we use the letter $A$ in the subscript, say, in $\varphi_{NA}$ to indicate that the conditions at the root are Neumann (standard) while at all other vertices they are Arbitrary). In this notation $\phi_N=\varphi_{NN}$ and  $\phi_D=\varphi_{DN}$. We recall the following reduction formulas from \cite[Theorem 2.1]{LawP}: Let $T^{(1)}$ and $T^{(2)}$ be two trees having exactly one common vertex $\widehat{v}$ viewed as the root of the trees $T^{(1)}$, $T^{(2)}$ and $T:=T^{(1)}\cup T^{(2)}$. We impose arbitrary self-adjoint boundary conditions at all vertices but $\widehat{v}$, and consider the Dirichlet and standard conditions at $\widehat{v}$. Let $\varphi_{DA}(\cdot)$, respectively, $\varphi_{NA}(\cdot)$ denote the characteristic function for the tree $(\cdot)$ equipped with the Dirichlet, respectively, standard condition at $\widehat{v}$. Then
\begin{equation}\label{lawpf}\begin{split}
\varphi_{NA}(T)&=\varphi_{DA}(T^{(1)})\varphi_{NA}(T^{(2)})+\varphi_{NA}(T^{(1)})\varphi_{DA}(T^{(2)}),\\
\varphi_{DA}(T)&=\varphi_{DA}(T^{(1)})\varphi_{DA}(T^{(2)}).\hskip6cm\Diamond\end{split}
\end{equation}
\end{remark}

We now proceed with several elementary examples.

\begin{example}\label{ex:p=2} The characteristic matrices for a segment (see $p=2$ in Figure \ref{lowp}) are 
\begin{equation}\label{PHIp2}
\Phi_{DD}=\begin{bmatrix}1&0\\0&s(\ell)\end{bmatrix},
\Phi_{DN}=\begin{bmatrix}1&0\\c'(\ell)&s'(\ell)\end{bmatrix},
\Phi_{ND}=\begin{bmatrix}0&1\\c(\ell)&s(\ell)\end{bmatrix},
\Phi_{NN}=\begin{bmatrix}0&1\\c'(\ell)&s'(\ell)\end{bmatrix},
\end{equation}
where $DD$ means the Dirichlet condition at $v_0$ and at $v_1$, $DN$ means the Dirichlet condition at $v_0$ and Neumann at $v_1$, etc. Thus, in this notation $\Phi_D(\lambda)=\Phi_{DN}$ and $\Phi_N(\lambda)=\Phi_{NN}$.
\hfill$\Diamond$\end{example}

\begin{example}\label{ex:3.1}  The matrices $\Phi_D$ and $\Phi_N$ for the tree pictured in item $p=3(a)$  of Figure \ref{lowp} with equal potentials are, respectively, the following two $4\times 4$-matrices, 
\begin{equation}
\label{PHND}
\Phi_D(\lambda)=\begin{bmatrix} 1&0&0&0\\
0&1&0&0\\
c'(\ell)&0&s'(\ell)&0\\
0&c'(\ell)&0&s'(\ell)\end{bmatrix}, \quad
\Phi_N(\lambda)=\begin{bmatrix}1&-1&0&0\\0&0&1&1\\
c'(\ell)&0&s'(\ell)&0\\
0&c'(\ell)&0&s'(\ell)\end{bmatrix}.
\end{equation}
Indeed, the Dirichlet conditions at $v_0$ and Neumann conditions at $v_1$, $v_2$ could be written as $y_1(0)=0$, $y_2(0)=0$, $y_1'(\ell)=0$, $y_2'(\ell)=0$
leading by \eqref{sc} to $a_1=0$, $a_2=0$, $a_1c'(\ell)+b_1s'(\ell)=0$, $a_2c'(\ell)+b_2s'(\ell)=0$
while the Kurchhoff's condition at $v_0$ and Neumann conditions at $v_1$, $v_2$ could be written as $y_1(0)-y_2(0)=0$, $y_1'(0)+y_2'(0)=0$, $y_1'(\ell)=0$, $y_2'(\ell)=0$
leading by \eqref{sc} to $a_1-a_2=0$, $b_1+b_2=0$, $a_1c'(\ell)+b_1s'(\ell)=0$, $a_2c'(\ell)+b_2s'(\ell)=0$ and so the respective systems of $2g=4$ equations for the unknown vector $\bu=(a_1,a_2,b_1,b_2)^\top\in\bbC^{2g}$ have matrices given in \eqref{PHND}. We calculate $\phi_D(\lambda)=(s'(\ell))^2$ and $\phi_N(\lambda)=-2c'(\ell) s'(\ell)$. In particular, $\phi_D(\lambda)=\phi_N(\lambda)=0$ if and only if $s'(\ell)=0$; if this is the case then the kernels of the matrices $\Phi_D(\lambda)$ and $\Phi_N(\lambda)$ in \eqref{PHND} have a nonzero intersection. This will be used in the proof of Lemma \ref{rem:intersection} below.
\hfill$\Diamond$\end{example}

\begin{example}\label{ex:p=2b}
The matrices $\Phi_D$ and $\Phi_N$ for the tree pictured in item $p=3b$ of Figure \ref{lowp} are
\begin{equation}\label{PHI3b}
\Phi_D(\lambda)=\begin{bmatrix}1&0&0&0\\ c(\ell)&-1&s(\ell)&0\\c'(\ell)&0&s'(\ell)&-1\\0&c'(\ell)&0&s'(\ell)\end{bmatrix},
\Phi_N(\lambda)=\begin{bmatrix}0&0&1&0\\c(\ell)&-1&s(\ell)&0\\c'(\ell)&0&s'(\ell)&-1\\0&c'(\ell)&0&s'(\ell)\end{bmatrix};
\end{equation}
 the characteristic functions are $\phi_D(\lambda)=-\big(s'(\ell))^2+s(\ell) c'(\ell)\big)$ and $\phi_N(\lambda)=-c'(\ell)\big(c(\ell)+s'(\ell)\big)$, where, to simplify notation, we assumed that the potentials are equal on all edges.

 In Lemma \ref{rem:intersection} below we will use the fact that for the current example there are no values of $\lambda$ such that $\phi_D(\lambda)=\phi_N(\lambda)=0$. To see this, 
we note that the current example can be reduced to Example \ref{ex:p=2} where the graph is just one segment of length $2\ell$. The Dirichlet problem on the graph in item $p=3(b)$ is the DN-problem on the doubled segment while the Neumann problem is the NN-problem, and the eigenvalues of the DN- and NN-problems on an interval cannot 
coincide.
\hfill$\Diamond$\end{example}



\begin{example}\label{ex:2d} The matrices  $\Phi_D$, $\Phi_N$ for the tree pictured in item (b) of Figure \ref{snowfl} are 
 \[
 \begin{array}{c|ccccccccc}
 \Phi_D&a_1 & a_2 & a_3 & a_{31} & b_1 & b_2 & b_3 & b_{31} \\ \hline
 v_0 & 1& 0 & 0 & 0 & 0&0&0&0\\
  v_0 & 0& 1 & 0 & 0 & 0&0&0&0\\
  v_0 & 0& 0 & 1 & 0 & 0&0&0&0\\
  v_1 & c'(\ell)&0&0&0&s'(\ell)&0&0&0\\
    v_2 &0& c'(\ell)&0&0&0&s'(\ell)&0&0\\
    v_3&0 &0& c(\ell)&-1&0&0&s(\ell)&0\\
     v_3&0 &0& c'(\ell)&0&0&0&s'(\ell)&-1\\
      v_{31}&0&0 &0& c'(\ell)&0&0&0&s'(\ell)
 \end{array},\]
\[\begin{array}{c|ccccccccc}
 \Phi_N&a_1 & a_2 & a_3 & a_{31} & b_1 & b_2 & b_3 & b_{31} \\ \hline
 v_0 & 1& -1 & 0 & 0 & 0&0&0&0\\
  v_0 & 1& 0 & -1 & 0 & 0&0&0&0\\
  v_0 & 0& 0 & 0 & 0 & 1&1&1&0\\
  v_1 & c'(\ell)&0&0&0&s'(\ell)&0&0&0\\
    v_2 &0& c'(\ell)&0&0&0&s'(\ell)&0&0\\
    v_3&0 &0& c(\ell)&-1&0&0&s(\ell)&0\\
     v_3&0 &0& c'(\ell)&0&0&0&s'(\ell)&-1\\
      v_{31}&0&0 &0& c'(\ell)&0&0&0&s'(\ell)
 \end{array},\]
 where, to simplify notation, we assumed that the potentials are the same on all edges.
\hfill$\Diamond$\end{example}

Returning to the general discussion, let $\cT$ be a combinatorial  tree rooted at $v_0$ with an arbitrary number $d(v_0)$ of subtrees. In the discussion that follows we assume that the root is an interior vertex; the case when it is pendant is analogous.
To write down the respective characteristic matrices in the general case in more details  we will use the enumeration of the edges introduced in Remark \ref{rem:enume}. As mentioned above, we will insert in the boundary conditions the functions from \eqref{yj} with the unknown coefficients $a_j,b_j$ corresponding to the edges $e_1,\ldots,e_{d(v_0)}$,
$e_{11},\ldots,e_{1(d(v_1)-1)}$, \ldots, $e_{d(v_0)1},\ldots e_{d(v_0)(d(v_{d(v_0)})-1)}$,
$e_{111},\ldots, e_{11(d(v_{11})-1)}$, and so on, in this particular order. 

We introduce the following $(2g\times 1)$ vector $\bu$ that represents the
 unknown coefficients $a_j$ and $b_j$,
\begin{align}\label{ab}
&\bu:=\big(a_1,\ldots,a_{d(v_0)},a_{11},\ldots, a_{1,(d(v_1)-1)}, a_{21},\ldots,a_{d(v_0)1},\ldots,
a_{d(v_0)(d(v_{d(v_0)})-1)}, a_{111},\ldots,\\ \nonumber
& a_{11(d(v_{11})-1)},a_{121},\ldots,a_{12(d(v_{12})-1)}a_{131},\ldots, b_1,\ldots,b_{d(v_0)},b_{11},\ldots,b_{d(v_0)(d(v_{d(v_0)})-1)},\ldots\big)^\top\in\bbC^{2g\times1}.
\end{align}
The coefficients in \eqref{ab} will satisfy the system of $2g$ linear equations as there are $\sum_{v\in\cV}d(v)$ boundary conditions that one must satisfy. The characteristic matrices $\Phi_D$ and $\Phi_N$ would have different first $d(v_0)$ rows and equal remaining rows. Indeed, we write the Dirichlet conditions at $v_0$ as $y_1(0)=0$, \ldots, $y_{d(v_0)}(0)=0$ and the standard conditions at $v_0$ as $y_1(0)-y_2(0)=0$, \ldots, $y_1(0)-y_{d(v_0)}(0)=0$, $y_1'(0)+\ldots+y_{d(v_0)}'(0)=0$ while the standard conditions at each $v_{k_1}$, $k_1=1,\ldots,d(v_0)$, are written as $y_{k_1}(\ell)-y_{k_1k_2}(0)=0$ for $k_2=1,\ldots, d(v_{k_1})$ and $y_{k_1}'(\ell)-\sum_{k_2=1}^{d(v_{k_1})-1}y_{k_1k_2}'(0)=0$. In terms of $a$'s and $b$'s from \eqref{yj} this leads to the following $d(v_0)$ first equations: for the Problem D they are
 $a_1=0$, $a_2=0$, \ldots, $a_{d(v_0)}=0$, while for the Problem N they are $a_1-a_2=0$,\ldots, $a_1-a_{d(v_0)}=0$, $b_1+\ldots+b_{d_0}=0$. The next $d(v_{1})$ equations corresponding to the boundary conditions at $v_1$ involve the coefficients 
 $a_{1}, b_{1}$ and $a_{1k_2}$, $b_{1k_2}$ for $k_2=1,\ldots, d(v_1)-1$ and are the same for both Problems D and N. The next $d(v_{2})$ equations corresponding to the boundary conditions at $v_2$ involve the coefficients 
 $a_{2}, b_{2}$ and $a_{2k_2}$, $b_{2k_2}$ for $k_2=1,\ldots, d(v_2)-1$ and are the same for both Problems D and N. We continue inductively. 
 
 The matrices $\Phi_D$ and $\Phi_N$ are represented in the two tables in Appendix \ref{TT}. In the tables, to simplify notation, we assumed that the potentials are equal on all edges. In the more general case when they are not, $c(\ell)$, $s(\ell)$, $c'(\ell)$, $s'(\ell)$ must be equipped with subscripts indicating the number of the respective edge. The first $g$ columns in the tables correspond to the coefficients $a_j$'s in \eqref{yj} while the remaining $g$ columns to the coefficients $b_j$'s. The order of the columns of the matrices corresponds to the enumeration of the edges described in Remark \ref{rem:enume} while the numbering of the rows corresponds to our usual enumeration of the vertices: first $d(v_0)$ rows correspond to the boundary conditions at $v_0$, the next $d(v_1)$ rows to the boundary conditions at $v_1$, and so on, the next $d(v_{d(v_0)})$ rows correspond to the boundary conditions at $v_{d(v_0)}$, the next $d(v_{11})$ rows correspond to the boundary conditions at $v_{11}$, then at $v_{12}$, and so on, then at $v_{1(d(v_1)-1)}$, and so forth. We abbreviate $d_0=d(v_0)$, $d_1=d(v_1)$, etc.

\begin{remark}\label{rem:split}
 For future use, it is convenient to split the $(2g\times 2g)$-matrices $\Phi_D=\Phi_D(\lambda)$ and $\Phi_N=\Phi_N(\lambda)$ into the blocks $[\Phi^{ij}]_{i,j=1}^2$ where  the block $\Phi^{11}$ is of the size 
 $d_0\times d_0$, the block $\Phi^{12}$ is of the size $d_0\times(2g-d_0)$,  the block $\Phi^{21}$ is of the size $(2g-d_0)\times d_0$
 and  the block $\Phi^{22}$ is of the size $(2g-d_0)\times(2g-d_0)$. For the respective blocks of the matrices $\Phi_D$ and $\Phi_N$ we clearly have,
 \begin{equation}\label{blocks1}\begin{split}
 \Phi_D^{12}&=0,\, \Phi_N^{21}=\Phi_D^{21}=:\Phi^{21},\, \Phi_N^{22}=\Phi_D^{22}=:\Phi^{22},\,  \Phi_D^{11}  =\diag\{1,1,\ldots 1\}, \\
 \Phi_D&=\begin{bmatrix}\Phi_D^{11}&0_{d_0\times(2g-d_0)}\\
 \Phi^{21}&\Phi^{22}\end{bmatrix}, \, \Phi_N^{11}=\begin{bmatrix}1 &-1&\dots & 0 \\
  \dots & \dots & \dots & \dots \\ 1&0&\dots&-1\\
 0 & 0 &\dots & 0 \end{bmatrix}, \\
 \Phi_N^{12}&=\begin{bmatrix}0&\dots&0&0&\dots&0&0&\dots&0\\
 \dots&\dots&\dots& \dots&\dots&\dots&\dots&\dots&\dots\\
 0&\dots&0&0&\dots&0&0&\dots&0\\
0&\dots&0& 1&\dots&1&0&\dots&0\\\end{bmatrix},
 \end{split}
 \end{equation}
 where in the last block the left- and -rightmost zero blocks are of the size $d_0\times(g-d_0)$ while
 the row-vector $\begin{bmatrix}1&\dots&1\end{bmatrix}$ is of the size $1\times d_0$. For future use, we remark that 
 \begin{equation}\label{blocks2}
 \phi_D(\lambda)=\det\Phi_D(\lambda)=\det\Phi_D^{11}\cdot\det\Phi^{22}=\det\Phi^{22}.
 \end{equation}
since the matrix $\Phi_D(\lambda)$ is lower block-triangular. \hfill$\Diamond$\end{remark}

We associate the metric tree $T$ with the combinatorial tree ${\mathcal T}$ described in Section \ref{sec:2}.  When indicated, we will impose the following assumption.
\begin{hypothesis}\label{hyp:two}
We assume that all edges have the same potential $q$; moreover, the potential is symmetric  with respect to the midpoints of the edges, that is, $q(\ell-x)=q(x)$ for almost all $x\in[0,\ell]$. 
\end{hypothesis}
An important consequence of Hypothesis \ref{hyp:two} is that, cf.\ \cite{CP} and \cite[Proposition 6.3.1]{MP},
\begin{equation}\label{lagr}
s'(\sl,\ell)=c(\sl,\ell),\, (c(\sl,\ell))^2-1=s(\sl,\ell)c'(\sl,\ell).
\end{equation}

The following theorem for Problem N is a version of  \cite[Theorem 5.2]{CP} which, in turn, is based on the results of \cite{vB}; we also refer to Theorem 6.3.3, Corollary 6.3.10 and Theorem 6.4.2 in \cite{MP}. 

\begin{theorem}\label{Theorem  3.1} 
 Let $T$ be a tree with $p\geq 2$ vertices.  Assume Hypothesis \ref{hyp:two}.
  Then the spectrum of 
  Problem N (see \eqref{3.1}--\eqref{3.4}, \eqref{3.7})  coincides with the set of zeros of its characteristic function that can be computed by the formula
\begin{equation}
\label{3.8} 
\phi_N(\lambda)=(s(\sqrt{\lambda}, \ell))^{-1}\psi(c (\sqrt{\lambda}, \ell)),
\end{equation} 
where $\psi(z)=\det(-zD+A)$ is defined in \eqref{2.2}    with 
 $A$ being the adjacency matrix of $T$ and
$D$  defined in \eqref{2.1}.
\end{theorem}

Under Hypothesis \ref{hyp:two}, if $\lambda$ is a zero of the function $\lambda\mapsto s(\sqrt{\lambda}, \ell)$ then it is a zero of the function  $\lambda\mapsto \psi(c (\sqrt{\lambda}, \ell))$, and thus the ratio in \eqref{3.8} is well-defined, see
\cite[Theorem 5.2]{CP}. Indeed, we recall from \cite[Section 6]{MP} general relations 
\begin{equation}\label{phind}
\phi_N(\lambda)=(s(\sqrt{\lambda},\ell))^{g-p}\psi(c(\sqrt{\lambda},\ell))\text{ and }\phi_D(\lambda)=(s(\sqrt{\lambda},\ell))^{g-p+1}\widehat{\psi}(c(\sqrt{\lambda},\ell)) 
\end{equation}
that hold for graphs that are even more general than trees.
In particular, formulas \eqref{phind} show that if we know the functions $\phi_N(\lambda)$ and $\phi_D(\lambda)$ then we know the functions $\psi$ and $\widehat{\psi}$, and vice versa.

Our next result deals with multiplicities of the eigenvalues of $L_N$; we refer to  \cite[Section 6.3]{MP} and to \cite{RS12} for discussions on this topic.

\begin{theorem}\label{Theorem 3.1prime} Assume Hypothesis \ref{hypo:stass}.  The multiplicity of any eigenvalue $\lambda_0\in\spec(L_N)$ is equal to the multiplicity of $\lambda_0$ as a zero of the function $\lambda\mapsto\phi_N(\lambda)$.
\end{theorem}

\begin{proof} The proof consists of two parts: 
First, we assume that Hypothesis \ref{hyp:two} holds. We fix an eigenvalue $\lambda_0\in\spec(L_N)$ and consider two cases.

{\em Case 1.\,} Assume that $s(\sqrt{\lambda_0},\ell)=0$. As it is well-known, $+ 1$ and $-1$ belong to the spectrum of the pencil $-zD+A$ and so $\psi(z)=(z^2-1)\psi_0(z)$, where $\psi_0$ is a polynomial such that $\psi_0(\pm1)\neq0$, see, e.g., \cite[Section 6.3]{MP0}. Using \eqref{3.8} and \eqref{lagr} we conclude that 
$\phi_N(\lambda)=c'(\sl,\ell)\psi_0(c(\sl,\ell))$ where $\psi_0(c(\sqrt{\lambda_0},\ell))\neq0$ since $(c(\sqrt{\lambda_0},\ell))^2-1=0$. Since the roots of the function $\sl\mapsto c'(\sl,\ell)$ are simple, $\lambda_0$ is a simple root of $\phi_N(\cdot)$. Thus, $0$ is a simple eigenvalue of $\Phi_N(\lambda)$ and so $\lambda_0$ is a simple eigenvalue of $L_N$ as $\ker(L_N-\lambda_0)$ and $\ker\Phi_N(\lambda_0)$ are isomorphic by the construction of the characteristic matrix $\Phi_N(\lambda)$.

{\em Case 2.\,} Assume that $s(\sqrt{\lambda_0},\ell)\neq0$.  For brevity, we introduce notation $z_0=c(\sqrt{\lambda_0},\ell))$.  We claim that \begin{equation}\label{KCLAIM}
\text{$\ker (L_N-\lambda_0)$ and $\ker(-z_0D+A)$ are isomorphic.}
\end{equation} As soon as the claim is proved, the required equality of the multiplicities follows because the matrix $-z_0D+A$ is self-adjoint and thus the dimension of its kernel, that is, the geometric multiplicity of $0\in\spec(-z_0D+A)$ is equal to its algebraic multiplicity, that is, to the multiplicity of $z_0$ as a zero of the function $z\mapsto\psi(z)=\det(-zD+A)$, and therefore, by \eqref{3.8}, to the multiplicity of $\lambda_0$ as a zero of the function $\lambda\mapsto\phi_N(\lambda)$.

To begin the proof of the claim in \eqref{KCLAIM}, we fix any $y=(y_j)_{j=1}^g\in\ker(L_N-\lambda_0)$ and define the vector $\iota(y):=(Y_v)_{v\in\cV}\in\bbC^{p\times 1}$ as follows. If $v$ is a pendant vertex with the incoming edge $e_j$, we set 
$Y_v=(s(\sqrt{\lambda_0},\ell))^{-1}y_j(\ell)$. If $v$ is an interior vertex with the incoming edge $e_j$ and the outgoing edges $e_k$, we set 
\begin{equation}\label{defYv}
Y_v=(s(\sqrt{\lambda_0},\ell))^{-1}y_j(\ell)=(s(\sqrt{\lambda_0},\ell))^{-1}y_k(0), \, k=2,\ldots,d(v).
\end{equation} If $v=v_0$ is the root with the outgoing edges $e_k$, we set $Y_{v_0}=(s(\sqrt{\lambda_0},\ell))^{-1}y_k(0)$, $k=1,\ldots,d(v_0)$. Since the eigenfunction $y=(y_j)_{j=1}^g$ satisfies the continuity conditions at all vertices, the vector $\iota(y)$ is well defined. The map $\iota$ is injective because $Y_v=0$ for each $v\in\cV$ would mean that $y_j$ solves the boundary value problem on each edge with Dirichlet conditions at both ends and therefore by the calculation of $\phi_{DD}$ in Example \ref{ex:p=2} one would had $s(\sqrt{\lambda_0},\ell)=0$ which is not the case. Since both subspaces in \eqref{KCLAIM} are finite dimensional, to prove that the injective map $\iota$ is an isomorphism between the two subspaces, it remains to show that $\iota$ maps $y$ into an element of $\ker(-z_0D+A)$.

To begin the proof of the assertion $(Y_v)_{v\in\cV}\in\ker(-z_0D+A)$, we consider any edge $e=[v',v'']$ outgoing from a vertex $v'$ and incoming to a vertex $v''$, and rewrite the respective eigenfunction $y_e=y_j$, $e=e_j$, as follows:
\begin{equation}\label{ye}
y_e(x)=s(\sqrt{\lambda_0},\ell)Y_{v'}c(\sqrt{\lambda_0},x)+\big(Y_{v''}-z_0Y_{v'}\big)s(\sqrt{\lambda_0},x),\, x\in[0,\ell]\approx e.
\end{equation}
Indeed, $y_e(0)=s(\sqrt{\lambda_0},\ell)Y_{v'}$ while $y_e(\ell)=s(\sqrt{\lambda_0},\ell)Y_{v''}$ which corresponds to the definition of $Y_v$ above. Furthermore, differentiating \eqref{ye} and using identities \eqref{lagr} gives
\begin{equation}\label{yeder}
y'_e(0)=Y_{v''}-z_0Y_{v'},\, y'_e(\ell)=-Y_{v'}+z_0Y_{v''}.
\end{equation}
We will now use the Kirchhoff condition at each vertex $v\in\cV$. We give the calculation only for the case when $v\neq v_0$ is an interior  vertex with an incoming edge $e_j=[v_1,v]$ and the outgoing edges $e_k=[v,v_k]$, $k=2,\ldots,d(v)$; all other cases are analogous. 
Then \eqref{yeder} yields
\begin{equation}\label{yeder2}\begin{split}
0&=y'_j(\ell)-\sum_{k=2}^{d(v)}y'_k(0)=-Y_{v_1}+z_0Y_{v}-\sum_{k=2}^{d(v)}\big(Y_{v_k}-z_0Y_v\big)\\
&=z_0d(v)Y_v-\sum_{k=1}^{d(v)}Y_{v_k}=-\big(-z_0d(v)Y_v+\sum_{v'\in\cV}a_{vv'}Y_{v'}\big),
\end{split}\end{equation}
since $A=(a_{vv'})_{v,v'\in\cV}$ is the adjacency matrix. Thus, $(-z_0D+A)\iota(y)=0$ as required. This completes the first part of the proof of the theorem.

In the second part of the proof we use Hypothesis \ref{hypo:stass} and make a homotopy reduction to the case of Hypothesis \ref{hyp:two} assumed in the first part. To begin, we introduce the following family of potentials,
\begin{equation}\label{qeta}
q(x,\eta)=\eta^2q_1(\ell-x)+\eta\big(q_1(x)-q_j(x)\big)+q_j(x),\, x\in e_j, j=1,\ldots,g, \eta\in[0,1],
\end{equation}
where $q_1$ is the potential on $e_1$. Clearly, $q(\cdot, 0)$ is the original potential from \eqref{3.1} while the potential $q(\cdot,1)$ satisfies Hypothesis \ref{hyp:two}. 

Let $L_N(\eta)$ be the operator with the potential $q(\cdot, \eta)$, denote by $n(\lambda,\eta)=\dim\ker(L_N(\eta)-\lambda)$ the (geometric) multiplicity of $\lambda\in\spec(L_N(\eta))$, and consider the eigenvalue curves $\lambda(\cdot)$ given by the non-equal eigenvalues of $L_N(\lambda)$ ordered such that 
$\lambda_1(\eta)=\ldots=\lambda_{n_1}(\eta)<\lambda_{n_1+1}(\eta)=\ldots=\lambda_{n_1+n_2}(\eta)<\ldots$; here we denoted $n_1=n(\lambda_1(\eta),\eta)$, $n_2=n(\lambda_{n_1+1}(\eta),\eta)$, etc.
Since $L_N(\lambda)$ is a holomorphic in $\eta$ family of type (A) of selfadjoint operators, the general perturbation theory from \cite[Section VII.3]{Kato} tells us that the eigenvalue curves are holomorphic and can intersect only at finitely many points. Without loss of generality we will assume that the intersection points are not at $\eta=0$ or $\eta=1$; otherwise, we perturb the potential by a small constant. In particular, it follows from the perturbation theory that $n(\lambda(1),1)=n(\lambda(0),0)$ for each eigenvalue curve $\lambda(\cdot)$. 

For any $\lambda_0\in\spec(L_N(\eta_0))$ so that $\phi_N(\lambda_0,\eta_0)=0$
 we further denote by $m(\lambda_0,\eta_0)$ the multiplicity of $\lambda_0$ as a zero of the function $\lambda\mapsto\phi_N(\lambda,\eta_0)$, where $\phi_N(\lambda,\eta)=\det\Phi_N(\lambda,\eta)$ is the characteristic function of the operator $L_N(\eta)$.
 In this notation, our objective is to show that $n(\lambda(0),0)=m(\lambda(0),0)$ for any eigenvalue curve $\lambda(\cdot)$. We claim that $m(\lambda(0),0)=m(\lambda(1),1)$. 
 
 Assuming the claim, we use the fact that $q(\cdot,1)$ satisfies Hypothesis \ref{hyp:two} and so $n(\lambda(1),1)=m(\lambda(1),1)$ as established in the first part of the proof. Combining this with the above mentioned equality $n(\lambda(1),1)=n(\lambda(0),0)$, we have $n(\lambda(0),0)=n(\lambda(1),1)=m(\lambda(1),1)=m(\lambda(0),0)$, as required.
 
 To prove the claim $m(\lambda(1),1)=m(\lambda(0),0)$, we will use  \cite[Theorem 9.1.1]{MP0} dealing with the multiplicities of zero curves of the analytic function $\phi_N(\lambda,\eta)$. Indeed, by this theorem, if $\lambda_0$ is a zero of the function  $\phi_N(\cdot, \eta_0)$ of multiplicity $m(\lambda_0,\eta_0)$ then there exist natural numbers $l,p_k,m_k$, $k=1,\ldots,l$, such that $\sum_{k=1}^lp_km_k=m(\lambda_0,\eta_0)$, the function $\phi_N(\cdot,\eta)$ for all $\eta$ near $\eta_0$ has $m(\lambda_0,\eta_0)$ zeros near $\lambda_0$ and the zero curves can be arranged into groups $\lambda_{kj}(\eta)$, $j=1,\ldots,p_k$, $k=1,\ldots,l$, such that they are pairwise different and each $\lambda_{kj}(\eta)$ is a root of $\phi_N(\cdot,\eta)$ of multiplicity $m_k$. Since the zeros of the function $\phi_N(\cdot,\eta)$ are the eigenvalues of the operator $L_N(\eta)$, the zero curves $\lambda_{kj}(\cdot)$ are precisely the eigenvalue curves $\lambda(\cdot)$ for $L_N(\cdot)$ that intersect at $(\lambda_0,\eta_0)$. Thus, the
 curves $\lambda_{kj}$ are holomorphic. Since $\lambda_{kj}$ can be represented by Puiseux series (9.1.1) in \cite{MP0} that uses $(\eta-\eta_0)^{1/p_k}$, one must have $p_k=1$ and therefore $j=1$. Also, it follows from \cite[Theorem 9.1.1]{MP0} that
 for each eigenvalue curve $\lambda(\cdot)=\lambda_{k1}(\cdot)$ we have
 $m(\lambda(\eta),\eta)=m_k$ for all $\eta\neq\eta_0$ near $\eta_0$, that is, $m(\lambda(\eta),\eta)$ is locally constant for $\eta\neq\eta_0$ (and changes to $m(\lambda_0,\eta_0)=\sum_{k=1}^lm_k$ at $\eta=\eta_0$). Going through all (finitely many, if any) intersection points along the curve $\lambda(\cdot)$ yields $m(\lambda(0),0)=m(\lambda(1),1)$.
\end{proof}


 We now discuss the Dirichlet problem (Problem D). In this case, as in Section \ref{sec:2},
 we can view  $T$ as a union of $d(v_0)$  subtrees   $T_1$, $T_2$, ..., $T _{d(v_0)}$ which have common vertex $v_0$, and pose spectral problems on each of the subtrees separately assuming that  the Dirichlet condition is imposed at the root $v_0$ while at all other   vertices we keep the standard boundary conditions. As a result, we obtain  $d(v_0)$ separate problems on the subtrees.  
 
 We denote by $\widehat{T}_{k}$ the tree obtained by removing the pendant vertex with the Dirichlet boundary condition (the root) and the edge  incident to it in $T_{k}$. 
 Let $\widehat{A}_{T_k}=A_{\widehat{T}_k}$  be the adjacency matrix of ${\widehat T}_{k}$, let
$\widehat{D}_{T_{k}}=\diag \{d(v_{{k},1}), d(v_{{k},2}), ..., d(v_{{k},p_{k}-1})\}$, where 
$d(v_{{k},j})$ is the degree of the vertex $v_{{k},j}$ in $T_{k}$ and $p_{k}$ is the number of vertices $\{v_0, v_{{k},1}, ..., v_{{k},p_{k}-1}\}$ in $T_{k}$. 
We further consider the polynomial $\widehat{\psi}_{k}(z)$ defined by 
\begin{equation}
\label{3.10}
\widehat{\psi}_{{k}}(z):=\det(-z\widehat{D}_{T_{k}}+\widehat{A}_{T_k}),\quad k=1,\dots,d(v_0).
\end{equation}
The first part of our next result is Theorem 6.4.2 of \cite{MP} adapted to the case of a tree with the Dirichlet condition at one of the  vertices while the proof of the last statement in the next theorem is analogous to the proof of Theorem \ref{Theorem 3.1prime}.

\begin{theorem}\label{Theorem 3.2} Assume Hypothesis \ref{hyp:two}. Let $T_{k}$, $k=1,\ldots,d(v_0)$, be the subtree of $T$ with one edge rooted at its pendant vertex $v_0$. Let the Dirichlet condition be imposed at the root and the standard conditions at all other vertices.  Then the spectrum of the Dirichlet problem (\ref{3.1})--(\ref{3.4}), (\ref{3.9})  on $T_{k}$ coincides with the set of zeros of its characteristic function that can be computed by the formula
\begin{equation}
\label{3.11}
\phi_{D,{k}}(\lambda)=\widehat{\psi}_{k}(c(\sqrt{\lambda},\ell)), \quad {k}=1,\dots, d(v_0),
\end{equation}
where  $\widehat{\psi}_{k}(z)$ is given in \eqref{3.10}. Furthermore, under Hypothesis \ref{hypo:stass} the multiplicity of an eigenvalue $\lambda_0\in\spec(L_D)$ is equal to the multiplicity $\lambda_0$ as a zero of the function $\lambda\mapsto\phi_D(\lambda)$.
 \end{theorem}


Here, the function
\begin{equation}
\label{3.12}
\phi_D(\lambda)=\prod_{k=1}^{d(v_0)}\phi_{D,k}(\lambda)
=\prod_{k=1}^{d(v_0)}\widehat{\psi}_{k}(c(\sl,\ell))=\prod_{k=1}^{d(v_0)}\det(-c(\sqrt{\lambda},\ell)\widehat{D}_{T_k}+\widehat{A}_{T_k})
\end{equation} 
is the characteristic function $\phi_D(\lambda)=\widehat{\psi}(c(\sl,\ell))$ of the Dirichlet problem  (\ref{3.1})--(\ref{3.4}), (\ref{3.9}) on the initial tree $T$, where 
\begin{equation}
\label{**}
\widehat{\psi}(z):=\det(-z\widehat{D}+\widehat{A})=
\prod_{k=1}^{d(v_0)}\widehat{\psi}_{k}(z)=
\prod_{k=1}^{d(v_0)}\det(-z\widehat{D}_{T_{k}}+\widehat{A}_{T_k}).
\end{equation}
Thus, the Dirichlet problem (unlike the Neumann) leads to the block-diagonalization of the normalized Laplacian.

\begin{remark} We mention the following relation between the characteristic functions of $T$ and the subtrees $T_{k}$, $k=1,\ldots,d(v_0)$, also cf.\ \cite[Proposition 5.4.3]{MP},
\begin{equation}\label{phindk}
\phi_N(\lambda)=\sum_{k=1}^{d(v_0)}\phi_{N,k}(\lambda)\big(\prod_{j=1, j\neq k}^{d(v_0)}\phi_{D,j}(\lambda)\big).
\end{equation}
This formula follows from  \eqref{lawpf} by induction.
\hfill$\Diamond$\end{remark}

Hypothesis  \ref{hyp:two} trivially holds for the identically zero potential and thus, cf.\ Remark \ref{CH}, 
\begin{equation}\label{checknd}
\check{\phi}_N(\lambda)=\sl(\sin(\sl\ell))^{-1}\psi(\cos(\sl\ell)) \text{ and } \check{\phi}_D(
\lambda)=\widehat{\psi}(\cos(\sl\ell))\end{equation} by Theorems \ref{Theorem 3.1} and \ref{Theorem 3.2}. The following result is an asymptotic version of Theorems \ref{Theorem 3.1} and \ref{Theorem 3.2} that holds without assuming Hypothesis  \ref{hyp:two}.

\begin{lemma}\label{lem:phipsi}
Assume Hypothesis \ref{hypo:stass} and impose standard boundary conditions.
Then
\begin{align}\label{nass}
\phi_N(\lambda)&=\sl(\sin(\sl\ell))^{-1}\psi(\cos(\sl\ell))+O(1),\, \lambda>0, \lambda\to+\infty,\\
\phi_D(\lambda)&=\widehat{\psi}(\cos(\sl\ell))+O(1/\sl),\, \lambda>0, \lambda\to+\infty.
\label{dass}
\end{align}
\end{lemma}
\begin{proof} As it is well-known, see \cite{Mar,MP0}, under Hypothesis \ref{hypo:stass} for all $\lambda>0$ one has
\begin{equation}\label{reprform}\begin{split}
s_j(\sl,\ell)&=\frac{\sin\sqrt{\lambda}\ell}{\sqrt{\lambda}}- \frac{\frac12\int_0^{\ell}q_j(x)dx\cos(\sl\ell)}{\lambda} +\frac{g_1(\sqrt{\lambda})}{\lambda},\\
s_j'(\sl,\ell)&=\cos\sqrt{\lambda}\ell+ \frac{\frac12\int_0^{\ell}q_j(x)dx\sin(\sl\ell)}{\sqrt{\lambda}} +\frac{g_2(\sqrt{\lambda})}{\sqrt{\lambda}},\\
c_j(\sl,\ell)&=\cos\sqrt{\lambda}\ell+ \frac{\frac12 \int_0^{\ell}q_j(x)dx\sin(\sl\ell)}{\sqrt{\lambda}} +\frac{g_3(\sqrt{\lambda})}{\sqrt{\lambda}},\\
c_j'(\sl,\ell)&=-\sqrt{\lambda}\sin\sqrt{\lambda}\ell+ \frac12\int_0^{\ell}q_j(x)dx\cos(\sl\ell) +g_4(\sqrt{\lambda}),
\end{split}\end{equation}
where $g_1,\ldots, g_4$, are some bounded on $\R_+$ functions. Indeed, to verify, say, the first formula, one uses \cite[Corollary 12.2.10]{MP0} with $n=0$: formula (12.2.22) of \cite{MP0} gives
\[s_j(\sl,\ell)=\frac{\sin\sl\ell)}{\sl}+K(\ell,\ell)\frac{\cos(\sl\ell)}{\lambda}+\int_0^\ell\frac{\partial K(\ell,t)}{\partial t}\frac{\cos(\sl t)}{\lambda}\,dt,\]
where $K$ is the function introduced in \cite[Theorem 12.2.9]{MP0}. Applying Cauchy-Swartz inequality in the last integral and using that $\partial K(\ell,t)/\partial t\in L^2(0,\ell)$ by the last statement of \cite[Corollary 12.2.10]{MP0} yields the desired representation of $s_j(\sl,\ell)$; the other formulas in \eqref{reprform}
are dealt with analogously. We now recall that the characteristic function $\phi_N(\lambda)=\det\Phi_N(\lambda)$ polynomially depends on the quantities in the left-hand side of formulas \eqref{reprform} while the characteristic function $\check{\phi}_N(\lambda)$ corresponding to the identically zero potential  polynomially depends on the first terms in the right-hand side of formulas \eqref{reprform}. This yields $\phi_N(\lambda)=\check{\phi}_N(\lambda)+O(1)$ as $\lambda\to+\infty$. Analogously, as noted in Remark \ref{phinphid}, $\phi_D(\lambda)$ polynomially depends on
$s_j(\sl,\ell)$, $s'_j(\sl,\ell)$, $c_j(\sl,\ell)$ (but not on $c'_j(\sl,\ell)$) yielding
$\phi_D(\lambda)=\check{\phi}_D(\lambda)+O(1/\sl)$. Using formula \eqref{checknd} completes the proof. \end{proof}

The next result shows that if the operators $L_D$ and $L_N$ on $T$ share an eigenvalue then they must also share an eigenfunction.
\begin{lemma}\label{rem:intersection}
Let $T$ be a tree rooted at $v_0$. We assume Hypothesis \ref{hypo:stass} and impose standard conditions at all vertices
but the root and consider the operators $L_D$ and $L_N$ in $L^2(T)$ corresponding to the Dirichlet and standard condition at the root. Then $\lambda\in\spec(L_D)\cap \spec(L_N)$ is a common zero of the functions $\phi_D$ and $\phi_N$ if and only if 
\begin{equation}\label{intker}
\ker(L_D-\lambda)\cap\ker(L_N-\lambda)\neq\{0\}.\end{equation}
\end{lemma}
\begin{proof}  The ``if'' part is trivial; so, let us {\em assume} that $\lambda$ is such that $\phi_D(\lambda)=\phi_N(\lambda)=0$ and prove \eqref{intker}. We will use induction by the number of edges in $T$. Two cases with $g=2$ are discussed in
 Examples \ref{ex:3.1} and \ref{ex:p=2b}. In the second example our assumption on $\lambda$ never holds while for the first example we have seen that
 if it does hold then we have $\ker(\Phi_D(\lambda))\cap\ker(\Phi_N(\lambda))\neq\{0\}$ for the matrices in \eqref{PHND}.  Since the kernels of the matrices are isomorphic to the kernels of the operators $L_D-\lambda$ and $L_N-\lambda$, we conclude that \eqref{intker} does hold for $g=2$.
 
The induction step has several cases. {\em First}, let us assume that the root $v_0$ is pendant. We will apply the reduction formula \eqref{lawpf} for the graph $T$, the graph $T^{(1)}=e_1$, where $e_1$ is the edge of $T$ incident to $v_0$ and $v_1$, and the graph $T^{(2)}$ consisting of all other edges of $T$, and for the common vertex $\widehat{v}=v_1$ of $T^{(1)}$ and $T^{(2)}$. We impose the standard boundary conditions at all vertices of $T^{(2)}$ and apply \eqref{lawpf} twice: first, we impose the Dirichelt condition at $v_0$, and, second, we impose the Neumann condition at $v_0$. In notation of Remark \ref{rem:LawPiv} we have $\phi_D(T)=\varphi_{D(v_0)N}(T)$ in the first case and $\phi_N(T)=\varphi_{N(v_0)N}(T)$ in the second case. Then our assumption on $\lambda$ and the first equation in \eqref{lawpf} give
\begin{equation}\label{sysphi}
\begin{split}
0&=\varphi_{D(v_0)N}(T)=\varphi_{D(v_0)D(v_1)}(T^{(1)})\phi_N(T^{(2)})+
\varphi_{D(v_0)N(v_1)}(T^{(1)})\phi_D(T^{(2)}),\\
0&=\varphi_{N(v_0)N}(T)=\varphi_{N(v_0)D(v_1)}(T^{(1)})\phi_N(T^{(2)})+
\varphi_{N(v_0)N(v_1)}(T^{(1)})\phi_D(T^{(2)}).
\end{split}
\end{equation}
Formulas \eqref{PHIp2} in Example \ref{ex:p=2} show that
\begin{equation*}\begin{split}
\varphi_{D(v_0)D(v_1)}(T^{(1)})&=s(\ell),\, \varphi_{N(v_0)D(v_1)}(T^{(1)})=c(\ell),\\
\varphi_{D(v_0)N(v_1)}(T^{(1)})&=s'(\ell),\, \varphi_{N(v_0)N(v_1)}(T^{(1)})=c'(\ell).
\end{split}\end{equation*}
Since the Wronskian of the solutions $c(\sl,\cdot)$ and $s(\sl,\cdot)$ is nonzero, the homogenous system \eqref{sysphi} has only a trivial solution for the unknowns $\phi_D(T^{(2)})=0$ and $\phi_N(T^{(2)})=0$. But the number of edges of $T^{(2)}$ is $g-1$ and so, by the induction assumption we then conclude that \eqref{intker} holds for the operators $L_D(T^{(2)})$ and $L_N(T^{(2)})$ on the graph $T^{(2)}$. We now select  the common eigenfunction of the operators $L_D(T^{(2)})$ and $L_N(T^{(2)})$. The eigenfunction is supported on $T^{(2)}$; we extend it by zero to all of $T$. The resulting function will be the common eigenfunction of $L_D$ and $L_N$ in $L^2(T)$ and thus \eqref{intker} is proved for the {\em first} case.

The {\em second} case is when the root $v_0$ of $T$ is not a pendant vertex. Recall that $T_{k}$, $k=1,\ldots,d(v_0)$, are the subtrees of $T$ each rooted at $v_0$ with the edge $e_{k}$ incident to $v_0$. We use the equality $\phi_D(\lambda)=0$ and fix an eigenfunction $y$ of $L_D$ in $L^2(T)$ satisfying the Dirichlet condition $y_{k}(0)=0$ for each $k=1,\ldots, d(v_0)$. Here and in what follows $y_{k}$ denotes the restriction of $y$ to the edge $e_{k}$ while $y\big|_{T_{k}}$ will denote the restriction on $y$ to the entire subtree $T_{k}$. Since $y$ is an eigenfunction, there is at least one value of $k$ such that $y\big|_{T_{k}}$ is not identically zero.

The second case has two sub-cases: (a) Let us consider the sub-case when for at least one subtree $T_{k}$, say, for $T_1$, the restriction of $y\big|_{T_1}$ is not identically zero and satisfies $y_1'(0)=0$. Then $y\big|_{T_1}$ is an eigenfunction of both $L_D(T_1)$ and $L_N(T_1)$ and thus $\lambda$ belongs to the intersection of the spectra of $L_D(T_1)$ and $L_N(T_1)$. By the induction assumption (or by the first case just proved since $v_0$ is a pendant root of $T_1$) we conclude that \eqref{intker} holds for the operators $L_D(T_1)$ and $L_N(T_1)$. As above, we take the common eigenfunction of the two operators and extend it by zero from $T_1$ to $T$ thus proving \eqref{intker} for $L_D(T)$ and $L_N(T)$.

(b) Let us now consider the sub-case when for at least two values of $k$, say, for $k=1$ and $k=2$, we have $y'_1(0)\neq0$ and $y'_2(0)\neq0$. We pick a nonzero vector $(a_1,a_2)\in\bbC^2$ such that $a_1y_1'(0)+a_2y'_2(0)=0$ and define on $T$ the function $Y$ by the rule $Y(x)=a_1y\big|_{T_1}(x)$ for $x\in T_1$, $Y(x)=a_2y\big|_{T_2}(x)$ for $x\in T_2$, and $Y(x)=0$ for $x\notin T_1\cup T_2$.
Clearly, $Y$ is an eigenfunction of $L_D$ because $Y_{k}(0)=0$ for all $k=1,\ldots,d(v_0)$ and $Y$ is an eigenfunction of $L_N$ because $\sum_{k=1}^{d(v_0)}Y'_{k}(0)=0$ by the choice of $a_1$, $a_2$. Thus, \eqref{intker} is proved.
We can not have exactly one nonzero derivative at $v_0$ as Kirchhoff's condition must hold. \end{proof}

\section{Attaching a lead to a compact metric trees}\label{sec:scattering}

In this section we consider a graph $T_\infty$ obtained by attaching an infinite lead,  $e_0=[0,+\infty)$, at the root $v_0$ of a compact equilateral tree $T$ with $g$ edges described in the previous section, see, e.g., Figure \ref{snowfl} (c) and (d). We direct the edge $e_0$ away from $v_0$ and assume that the potential $q_0$ on $e_0$ is identically zero. Thus we have the equations
\begin{equation}
\label{4.1}
-y_j^{\prime\prime}(x)+q_j(x)y_j(x)=\lambda y_j(x),\,  j=1,\ldots,g,\, x\in [0,\ell],
\end{equation}
on all finite edges along with the equation
\begin{equation}
\label{4.2}
-y_0^{\prime\prime}(x)=\lambda y_0(x),\, x\in e_0:= [0,\infty),
\end{equation}
on the edge $e_0$.
We endow the Sturm--Liouville equation~\eqref{4.1}--\eqref{4.2} on $T_\infty$ with the standard conditions at all vertices including $v_0$.


We introduce the usual Hilbert spaces 
\[
L^2(T_\infty):=L^2(0,\infty)\oplus \bigoplus\limits_{j=1}^g L^{2}(0,\ell),
\quad H^2(T_\infty):=H^2(0,\infty)\oplus \bigoplus\limits_{j=1}^g H^{2}(0,\ell)
\]
of square-integrable and Sobolev vector-valued functions $y=(y_j)_{j=0}^g$, and
 the operator $L_\infty$ in $L^2(T_\infty)$ related to the boundary value problem
\eqref{4.1}--\eqref{4.2} equipped with the standard conditions at all vertices including the root. The operator $L_\infty$ acts as follows,
 \begin{equation} 
 \label{4.03}
L_\infty(y_j)_{j=0}^{g}=(-y^{\prime\prime}_{j}+q_{j}y_{j})_{j=0}^g,  \text{ where $q_0(x)=0$ for $x\in e_0$,}
\end{equation}
 and its domain is given by the formula
 \begin{equation} 
 \label{4.3}\begin{split}
\dom(L_\infty):=\big\{(y_j)_{j=0}^g &\in H^2(T_\infty):\, y_j(\ell)-y_k(0)=0,\\ &\qquad
y_j'(\ell)-\sum_{k: e_k\twoheadleftarrow v}y_k'(0)=0, \, e_k\twoheadleftarrow v, e_j\twoheadrightarrow v,
\text{ for all } v\in V \big\};
\end{split}
\end{equation}
here $e_k\twoheadleftarrow v$ means that the edge $e_k$ is outgoing from while
$e_j\twoheadrightarrow v$ means that the edge $e_j$ is incoming to
 a vertex $v$. 
 \begin{remark}\label{rem:domLinfty}
 If $v_0$ is the root of $T$ then the boundary conditions in \eqref{4.3} at $v_0$ are given by $y_0(0)-y_k(0)=0$ for $k=1,\ldots, d(v_0)$ and $\sum_{k=0}^{d(v_0)}y'_k(0)=0$. In particular, if $d(v_0)=1$, that is, there is only one edge,  $e_1$, in $T$ outgoing from $v_0$ and connecting $v_0$ to $v_1$, then $e_1\cup e_0$ could be viewed as a new lead attached to $v_1$.
Under the given boundary condition the derivatives of the functions from $\dom(L_\infty)$ are continuous at $v_0$. Thus, the operator $L_\infty$ is the same as the operator $L'_\infty$ related to the graph $T'_\infty$ obtained by attaching the new lead at $v_1$
  to the graph $T'$ rooted at $v_1$ and obtained from $T$ by deleting $e_1$ and $v_0$. Thus, without loss of generality, we may assume in what follows that $d(v_0)\ge2$.
 \hfill$\Diamond$\end{remark}
We identify the spectrum of the operator $L_\infty$ in \eqref{4.03}--\eqref{4.3} with the spectrum of the boundary problem \eqref{4.1}--\eqref{4.2} equipped with the standard conditions at all vertices. We recall the following result from \cite[Theorem 3.1]{MuP}.

\begin{theorem} \label{thm: 4.1}
Under Hypothesis \ref{hypo:stass} and $q_0=0$ on $e_0$, the operator $L_\infty$ defined in \eqref{4.03}--\eqref{4.3} is self-adjoint and
bounded from below  in $L^2(T_\infty)$.
Furthermore, $\spec_{\mathrm{ess}}(L_\infty)=[0,\infty)$.
\end{theorem}

Our next objective is to study the solution $y=y(\lambda, \cdot)=(y_j(\cdot))_{j=0}^g$ to  the Sturm--Liouville equations \eqref{4.1}--\eqref{4.2} on $T_\infty$ that satisfies the standard conditions at all vertices (including $v_0$) and such that the restriction of the solution to the lead $e_0$ attached at the root $v_0$ of $T$ is given by the formula
\begin{equation}
\label{*}
y_0(\lambda,x)=
\phi_D(\lambda)\check{c}(\sqrt{\lambda},x)
-\phi_N(\lambda)\check{s}(\sqrt{\lambda},x), \, \sl\in\bbC,\, x\in e_0.
\end{equation}
 Here we use the ``check''-notation introduced in Remark \ref{CH}:  $\check{s}$, $\check{c}$ are the solutions to the equation $-y''=\lambda y$ on $e_0$ satisfying $\check{s}(\sqrt{\lambda},0)=0$, $\check{s}'(\sqrt{\lambda},0)=1$, $\check{c}(\sl,0)=1$, 
 $\check{c}'(\sl,0)=0$, 
 that is, $\check{s}(\sqrt{\lambda},x)=\frac1{\sqrt{\lambda}} \sin(\sqrt{\lambda} x)$ when $\lambda\neq0$ and $\check{s}({0},x)=x$ while $\check{c}(\sqrt{\lambda},x)=\cos(\sqrt{\lambda}x)$ when $\lambda\neq0$ and $\check{c}({0},x)=1$.


\begin{lemma}\label{phiphi} 
Assume Hypothesis \ref{hypo:stass} and $q_0=0$ on $e_0$ and fix any $\lambda\in\bbC$. Then:
\begin{itemize}
\item[(i)]  There exists a solution $y=y(\lambda,\cdot)=(y_j(\cdot))_{j=0}^g$ to \eqref{4.1}--\eqref{4.2} on $T_\infty$ 
that satisfies the standard conditions at all vertices including $v_0$ and has the form \eqref{*} on the lead $e_0$.
\item[(ii)]
If the intersection $\ker(L_D-\lambda I_{L^2(T)})\cap\ker(L_N-\lambda I_{L^2(T)})$  is trivial then the  solution  in (i)  is unique up to a scalar multiple.
 \item[(iii)] If the intersection in (ii) is not trivial then the number of linearly independent solutions  in (i) is equal to the dimension of the intersection.
 \item[(iv)] For any nonzero solution $y=(y_j)_{j=0}^g$ to   \eqref{4.1}--\eqref{4.2} satisfying the standard conditions at all vertices equation \eqref{*} necessarily holds.\end{itemize}
\end{lemma}
\begin{proof} For definiteness, we give the proof for the case when $v_0$ is an interior vertex; the case when it is pendant is analogous, cf.\ also Remark \ref{rem:domLinfty}.
By Lemma \ref{rem:intersection} the intersection in (ii) is nontrivial if and only if $\lambda\in\spec(L_D)\cap\spec(L_N)$, equivalently, $\phi_D(\lambda)=\phi_N(\lambda)=0$, and if this is the case then the operators $L_D$ and $L_N$ share $\dim\big(\ker(L_D-\lambda I_{L^2(T)})\cap\ker(L_N-\lambda I_{L^2(T)})\big)$ common linearly independent eigenfunctions. Each of the eigenfunctions gives a solution to 
\eqref{4.1}--\eqref{4.2} on $T_\infty$ that satisfies the standard conditions at all vertices (including $v_0$) with
the function $y_0$ being identically zero on $e_0$ so that  \eqref{*} holds. This proves parts (iii) of the lemma and part (i) in the case when $\phi_D(\lambda)=\phi_N(\lambda)=0$.

We are left with the main case of part (i) and part (ii) when at least one of the values $\phi_D(\lambda)$ and $\phi_N(\lambda)$ is nonzero. If this is the case, 
we temporarily introduce the following column- and row-vectors,
\[ \be_{d_0}:=(\delta_{d_0j})_{j=1}^{2g}= [0,\ldots, 0,1,0,\ldots,0]^\top\in\bbC^{2g\times1} 
\text{ and } \be'_1:=(\delta_{1j})_{j=1}^{2g} =[1,0\ldots,0]\in\bbC^{1\times 2g},\]
where $\delta_{ij}$ is the Kronecker's delta.
In the case that we are currently considering the column-vector 
$[\phi_D(\lambda)\,\,\, \phi_N(\lambda)\be_{d_0}^\top]^\top\in\bbC^{(2g+1)\times 1}$ is nonzero because the first entry of the vector is $\phi_D(\lambda)$ and the $(d_0+1)$-th entry is $\phi_N(\lambda)$.
Since $y_0(0)=\phi_D(\lambda)$ and $y_0'(0)=-\phi_N(\lambda)$ due to \eqref{*}, the standard boundary condition for the solution $y$ to \eqref{4.1}--\eqref{4.2} on $T_\infty$ at $v_0$ becomes inhomogeneous and reads $y_1(0)=\phi_D(\lambda)$, $y_1(0)-y_2(0)=0$, \dots, $y_1(0)-y_{d_0}(0)=0$ and $y_1'(0)+\ldots+y_{d_0}'(0)=\phi_N(\lambda)$, while the boundary conditions at all other vertices are standard and remain homogenous. We will use \eqref{yj} at each of the edges of $T$ and consider a {\em non-homogeneous} system of $2g+1$ linear equations for the unknowns listed in \eqref{ab}.
The boundary condition at $v_0$ gives $d_0+1$ equations $a_1=\phi_D(\lambda)$, $a_1-a_2=0$, \dots, $a_1-a_{d_0}=0$ and $b_1+\ldots+b_{d_0}=\phi_N(\lambda)$ while the boundary conditions at all other vertices give the remaining  $2g-d_0$ equations which are the same as in the construction of the characteristic matrix $\Phi_N(\lambda)$. As a result, we obtain a non-homogenous system of $2g+1$ equations
with $2g$ unknowns $\bu:=(a_1,\ldots,a_{11},\ldots,a_{111},\ldots, b_1,\ldots, b_{11},\ldots)^\top$, cf.\ \eqref{ab}. The non-homogeneous system itself, the $(2g+1)\times(2g)$-matrix $\Phi$ of the system and the extended $(2g+1)\times(2g+1)$-matrix $\widehat{\Phi}$ of the system are as follows,
\begin{equation}\label{psipsi}
\Phi \bu=\begin{bmatrix}\phi_D(\lambda)\\\phi_N(\lambda)\be_{d_0}\end{bmatrix},\, \Phi:=\begin{bmatrix}\be'_1\\\Phi_N(\lambda)\end{bmatrix},\,
\widehat{\Phi}:=\begin{bmatrix}\be'_1 & \phi_D(\lambda)\\
\Phi_N(\lambda)& \phi_N(\lambda)\be_{d_0}\end{bmatrix},
\end{equation}
We claim that \begin{equation}\label{rank}
\rank(\widehat{\Phi})=\rank(\Phi)=2g.\end{equation} As soon as the claim is proved, the Kronecker-Capelli Theorem gives the existence of a unique solution $\bu$ to the system of equations in \eqref{psipsi} and thus the existence of a unique solution $y$ to \eqref{4.1}--\eqref{4.2} on $T_\infty$ that satisfies the standard conditions at all vertices (including $v_0$) and equation \eqref{*}, thus finishing the proof of parts (i) and (ii) of the lemma.

To begin the proof of \eqref{rank}, we expand the determinant of $\widehat{\Phi}$ using the last column to obtain
\begin{equation}\label{det1}
\det\widehat\Phi=\phi_D(\lambda)\det\Phi_N(\lambda)+(-1)^{2g+1+d_0+1}(\phi_N(\lambda))\det\widetilde{\Phi}_N,
\end{equation}
where the $(2g\times 2g)$-matrix $\widetilde{\Phi}_N$ is obtained from the matrix ${\Phi}_N(\lambda)$ by crossing out its $d_0$-th row and adding $\be'_1$ as the first row. In particular, $\widetilde{\Phi}_N$ is a sub-matrix of $\Phi$. We now split $\widetilde{\Phi}_N=(\widetilde{\Phi}_N^{ij})_{i,j=1}^2$ as indicated in Remark \ref{rem:split}, cf.\ \eqref{blocks1}, to obtain 
\begin{equation*}
\widetilde{\Phi}_N^{11}=\begin{bmatrix}1&0&\dots&0\\
1&-1&\dots&0\\
\dots&\dots&\dots&\dots\\
1&0&\dots&-1\end{bmatrix}, \,
\widetilde{\Phi}_N^{12}=0, \widetilde{\Phi}_N^{21}=\Phi_D^{21}=\Phi_N^{21},
\, \widetilde{\Phi}_N^{22}=\Phi_D^{22}=\Phi_N^{22}.
\end{equation*}
 This and \eqref{blocks2} yield
 \begin{equation}\label{det3}
 \det\widetilde{\Phi}_N=\det\widetilde{\Phi}_N^{11}\cdot\det\widetilde{\Phi}_N^{22}
 =(-1)^{d_0-1}\det\Phi_D^{22}=(-1)^{d_0-1}\phi_D(\lambda).
 \end{equation}
 In turn, using this in \eqref{det1} gives $\det\widehat{\Phi}=0$ and so $\rank\widehat{\Phi}=\rank\Phi$. To show the second equality in \eqref{rank}, we remark that $\Phi_N(\lambda)$ is a $(2g\times 2g)$-sub-matrix of $\Phi$, and so if $\phi_N(\lambda)=\det\Phi_N\neq0$ then the second equality in \eqref{rank} does hold. In the case when $\phi_N(\lambda)=0$ but $\phi_D(\lambda)=\det\Phi_D\neq0$ we use the fact that $\Phi$ also has a $(2g\times 2g)$-sub-matrix $\widetilde\Phi_N$ which  is nonsingular by \eqref{det3} thus finishing the proof of \eqref{rank}.
 
 It remains to prove part (iv) of the lemma. We want to show that if $y=y(\lambda,\cdot)$ is a nontrivial solution to the equations \eqref{4.1}--\eqref{4.2} on $T_\infty$ that satisfies the standard boundary conditions at all vertices including the root, then the restriction $y_0$ of $y$ to $e_0$ is necessarily of the form \eqref{*}, that is, if  
\begin{equation}\label{yj0}
y_0(\lambda,x)=a_0\check{c}(\sqrt{\lambda},x)
+b_0\check{s}(\sqrt{\lambda},x)\end{equation} then $a_0=\phi_D(\lambda)$ and $b_0=-\phi_N(\lambda)$. First, we prove this in the case when $\lambda$ is such that at least one of the values $\phi_D(\lambda)$, $\phi_N(\lambda)$ is nonzero. Indeed, for $y_0$ from \eqref{yj0} we have $y_0(0)=a_0$ and $y'_0(0)=b_0$. Representing $y$ as in \eqref{yj} on $T$ and as in \eqref{yj0} on $e_0$ as above, we obtain for the $2g$ coefficients $\bu$ in \eqref{yj} recorded in \eqref{ab} an inhomogeneous system of $2g+1$ equations as in \eqref{psipsi} but modified such that $\phi_D(\lambda)$ is replaced by $a_0$ and $\phi_N(\lambda)$  is replaced by $-b_0$. The modified system in \eqref{psipsi} has a solution $\bu$. Then the
Kronecker-Capelli Theorem tells us that $\rank\widehat\Phi=\rank\Phi$ and therefore $\det\widehat\Phi=0$. Expanding the determinant of $\widehat{\Phi}$ as in \eqref{det1} gives $0=\det\widehat\Phi=a_0\det\Phi_N+(-1)^{d_0}(-b_0)\det\widetilde{\Phi}_N$. Using \eqref{det3}, this yields
\begin{equation}\label{abphi}
a_0\phi_N(\lambda)+b_0\phi_D(\lambda)=0.
\end{equation}
First, we consider the case when $\phi_N(\lambda)\neq0$.  If $b_0=0$ then $a_0=0$ and the modified system of equations in \eqref{psipsi} becomes $\Phi \bu=0$ which yields $\Phi_N\bu=0$ and so we must have $\bu=0$, which is not possible as the solution $y$ is nontrivial to begin with. This proves that if $\phi_N(\lambda)\neq0$ then $b_0\neq0$. We recall that equations \eqref{4.1}--\eqref{4.2} are homogeneous and pass from $y$ to the solution $-\frac{\phi_N(\lambda)}{b_0}y$ which is equal on $e_0$ to $-a_0\frac{\phi_N(\lambda)}{b_0}\check{c}(\sqrt{\lambda},x)
-\phi_N(\lambda)\check{s}(\sqrt{\lambda},x)$. Thus, without loss of generality we may assume that $b_0=-\phi_N(\lambda)$ in \eqref{yj0} from the start. But then \eqref{abphi} yields $a_0=\phi_D(\lambda)$ as required. The case $\phi_D(\lambda)\neq0$ is analogous to the case $\phi_N(\lambda)\neq0$. It remains to consider the case 
 when $\phi_D(\lambda)=\phi_N(\lambda)=0$ and prove that then $a_0=b_0=0$ in \eqref{yj0}. Since $\phi_D(\lambda)=0$, we know that $y$ satisfies the Dirichlet conditions $y_{k}(0)=0$, $k=1,\ldots,d(v_0)$, at $v_0$ viewed as the root of the graph $T$. Since $a_0=y_0(\lambda,0)$ and $b_0=y'_0(\lambda,0)$, and $y$ satisfies the standard (hence continuity) condition at $v_0$ viewed as a vertex of the graph $T_\infty$, this yields $a_0=0$. A similar argument using the Kirchhoff's condition  yields $b_0=0$.
\end{proof}

\begin{remark}\label{yT} The restriction $y_T$ to $T$ of the solution $y(\lambda,\cdot)$ to \eqref{4.1}--\eqref{4.2} on $T_\infty$ that satisfies the standard conditions at all vertices (including $v_0$) and equation \eqref{*} just constructed in Lemma \ref{phiphi} has the following remarkable properties: The function $y_T$ satisfies the Dirichlet condition $y_T(v_0)=0$ at $v_0$ if and only if $\lambda\in\spec(L_D)$ and the Kirchhoff's condition $\sum_{k=1}^{d(v_0)}y_{k}'(0)=0$ at $v_0$ if and only if $\lambda\in\spec(L_N)$. This follows from the fact that  $\lambda\in\spec(L_D)$ if and only if $\phi_D(\lambda)=0$ while  $\lambda\in\spec(L_N)$ if and only if $\phi_N(\lambda)=0$; for the solution to \eqref{4.1}--\eqref{4.2} satisfying \eqref{*} one has $y_0(0)=\phi_D(\lambda)$ and $y_0'(0)=-\phi_N(\lambda)$.
\hfill$\Diamond$\end{remark}




We will now utilize  the formalism of the classical scattering theory (see, e.g., \cite[Chapter 3]{Mar} or \cite[Chapter 5]{NZ} or \cite{RS3}). We follow the convention in \cite{Mar} and notice that if $\im(\sqrt{\lambda})>0$ then the function $\e^{\i\sqrt{\lambda}x}$ decays to zero while $\e^{-\i\sqrt{\lambda}x}$ growths to infinity as $x\to+\infty$.

Equation (\ref{*}) can be rewritten as follows,
\begin{equation}
\label{4.4}\begin{split}
y_0(\lambda,x)&=\frac{1}{2i\sqrt{\lambda}}
\big(e^{-i\sqrt{\lambda}x}(\phi_N(\lambda)+i\sqrt{\lambda}\phi_D(\lambda))\\&\qquad 
-e^{i\sqrt{\lambda}x}(\phi_N(\lambda)-i\sqrt{\lambda}\phi_D(\lambda))\big),\, \sl\in\bbC.\end{split}\end{equation}


Formula \eqref{4.4} is analogous to formula (5.2.14) in \cite{NZ} if we interpret $y_0$ from \eqref{*} as the ``regular'' solution. Indeed, computing the Wronskian $\cW$ of $y_0$ and the exponentially decaying for $\im\sqrt{\lambda}>0$ solution $e^{i\sqrt{\lambda}x}$ yields
$\cW(y_0, e^{i\sqrt{\lambda}x})\big|_{x=0}=\phi_N(\lambda)+i\sqrt{\lambda}\phi_D(\lambda)$.
 Building on this analogy, we introduce, following \cite[Lemma 3.1.5]{Mar}, the meromorphic in $\bbC$ (thanks to $q_0=0$) function of the argument $\sl\in\bbC$,
\begin{equation}\label{dfS}
S:\sl\mapsto \frac{E(-\sqrt{\lambda})}{E(\sqrt{\lambda})},
\end{equation}
where we define
\begin{equation}
\label{4.5}
E(\sqrt{\lambda}):=\phi_N(\lambda)+i\sqrt{\lambda}\phi_D(\lambda),\quad \sl\in\C.
\end{equation}
We call $S$ in \eqref{dfS} the {\em $S$-function} and $E$ in \eqref{4.5} {\em
the Jost function} since \eqref{4.4} could be re-written as
\begin{equation}\label{4.4.prime}
y_0(\lambda,x)=\frac{1}{2i\sqrt{\lambda}}\big( E(\sqrt{\lambda})e^{-i\sqrt{\lambda}x}-E(-\sqrt{\lambda})e^{i\sqrt{\lambda}x}\big). \end{equation}
It is clear that $\sqrt{\lambda}\mapsto E(\sqrt{\lambda})$ is an entire function  because the determinants of  $\Phi_D$ and $\Phi_N$ depend polynomially on the entire functions $c_j(\cdot,x)$, $s_j(\cdot,x)$, $c'_j(\cdot,x)$ and $s'_j(\cdot,x)$, see the discussion after equation \eqref{sc}. We summarize  the usual properties of the Jost function.


\begin{lemma}\label{propE} Assume Hypothesis \ref{hypo:stass} and that $\lambda\neq0$  and $\im\sqrt{\lambda}\ge0$.
Then
\begin{itemize}\item[(i)]
$\lambda$ is an eigenvalue of the operator $L_\infty$ in $L^2(T_\infty)$  if and only if $E(\sqrt{\lambda})=0$. 
\item[(ii)] 
All roots of the function $\sqrt{\lambda}\mapsto E(\sqrt{\lambda})$  are located either on the real line or on the positive imaginary half line. 
\item[(iii)]  $E(\sqrt{\lambda})=0$ for a real $\sqrt{\lambda}$ if and only if the positive $\lambda\in\spec(L_N)\cap\spec(L_D)$; the  $\lambda$ is an eigenvalue of $L_\infty$ embedded into its essential spectrum.
 \item[(iv)] $E(\sqrt{\lambda})=0$ for a pure imaginary $\sqrt{\lambda}$ with $\im(\sl)>0$ if an only if the negative $\lambda$
 is an isolated eigenvalue of $L_\infty$ in $L^2(T_\infty)$. In particular, every negative $\lambda\in\spec(L_N)\cap\spec(L_D)$ is an isolated negative eigenvalue of $L_\infty$.
\end{itemize}\end{lemma}
\begin{proof} Clearly, $\lambda$ is an eigenvalue of the operator $L_\infty$ in $L^2(T_\infty)$ if and only if $y_0(\lambda,\cdot)\in L^2(\R_+)$. If $\im\sqrt{\lambda}>0$ then formula \eqref{4.4.prime} shows that $y_0(\lambda,\cdot)\in L^2(\R_+)$ if and only if $E(\sl)=0$ because $e^{-i\sqrt{\lambda}x}$ growths to infinity when $x\to+\infty$.
If $\sqrt{\lambda}$ is real (and then $\lambda$ if positive) then formula \eqref{*} shows that $y_0(\lambda,\cdot)\in L^2(\R_+)$ if and only if $y_0(\lambda,\cdot)$ is identically $0$ which, in turn, happens if and only if $\phi_D(\lambda)=\phi_N(\lambda)=0$. For the real $\sqrt{\lambda}$ the latter condition is equivalent to $E(\sqrt{\lambda})=0$, finishing the proof of assertion (i). Assertion (ii) follows from (i) because $L_\infty$ is self-adjoint by Theorem \ref{thm: 4.1}. Assertion (iii)  holds because
 $\lambda\in\spec(L_N)\cap\spec(L_D)$ is equivalent to $\phi_D(\lambda)=\phi_N(\lambda)=0$. The first part of assertion (iv) follows from (i) while the second holds as $\phi_D(\lambda)=\phi_N(\lambda)=0$ trivially yields $E(\sqrt{\lambda})=0$.
\end{proof}
\begin{remark}\label{lambda=0}
Since  $\check{s}(0,x)=x$ and $\check{c}(0,x)=1$ and thus both functions do not belong to $L^2(e_0)$, we conclude from \eqref{*} that $0$ is an eigenvalue of $L_\infty$ (embedded into the essential spectrum)  if and only if $\phi_D(0)=\phi_N(0)=0$.
\hfill$\Diamond$\end{remark}

For the operator $L_\infty$ we now give an analogue of Theorem \ref{Theorem 3.1prime} (with a similar proof).
\begin{theorem}\label{lem:EvsL} Under Hypothesis \ref{hypo:stass}, the multiplicity of a negative eigenvalue $\lambda_0\in\spec(L_\infty)$ is equal to the multiplicity of $\sqrt{\lambda_0}$ as a root of the function $E(\cdot)$.
\end{theorem}
\begin{proof} As in Theorem \ref{Theorem 3.1prime}, the proof consists of two parts:
  First, we impose Hypothesis \ref{hyp:two} and fix a negative eigenvalue $\lambda_0\in\spec(L_\infty)$.
By Remark \ref{rem:domLinfty}, we assume that $v_0$, the root of $T$,  is interior so that $d(v_0)\ge2$. 

We begin by proving that $s(\sqrt{\lambda_0},\ell)\neq0$. We argue by contradiction and
assume that $s(\sqrt{\lambda_0},\ell)=0$. As in the proof of Case 1 in Theorem \ref{Theorem 3.1prime} we conclude that $\lambda_0$ is a simple root of $\phi_N(\cdot)$. Since, by Lemma \ref{propE}(i), $\sqrt{\lambda_0}$ is a root of the function $E(\sl)=\phi_N(\lambda)+i\sl\phi_D(\lambda)$ we must have
$\phi_D(\lambda_0)=0$. This is not possible by the following argument using Theorem \ref{Theorem 3.2}: Since $s(\sqrt{\lambda_0},\ell)=0$, we have 
$z_0:=c(\sqrt{\lambda_0},\ell)=\pm1$. However, $\widehat{\psi}(\pm1)\neq0$.
To show this, we re-write \eqref{detf} as 
 $\widehat{\psi}_k(z)=\det\Psi_{\widehat{T}_k}(z)-z\widehat{\psi}_{\widehat{T}_k}(z)$ where $\Psi_{\widehat{T}_k}(z)=-zD_{\widehat{T}_k}+A_{\widehat{T}_k}$ is the $\Psi$-matrix and $\widehat{\psi}_{\widehat{T}_k}$ is the $\widehat{\psi}$-function of the tree $\widehat{T}_k$ as defined in \eqref{2.2}. Since $\det\Psi_{\widehat{T}_k}(\pm1)=0$ as mentioned in Case 1 of the proof of Theorem \ref{Theorem 3.1prime}, we have $\widehat{\psi}_k(\pm1)=\mp\widehat{\psi}_{\widehat{T}_k}(\pm1)$, which is nonzero 
as easily follows by induction in $p$. This completes the proof of the claim.


Next, we take the $(p\times p)$ matrix $D_p^0=\diag\{1,0,\ldots,0\}$, denote, for brevity, $s_0=s(\sqrt{\lambda_0},\ell)$, $s=s(\sqrt{\lambda},\ell)$, $z_0=c(\sqrt{\lambda_0},\ell)$, $z=c(\sqrt{\lambda},\ell)$, and consider the matrix $-zD+i\sqrt{\lambda}sD_p^0+A$ for $\lambda$ near $\lambda_0$.
 Acting as in the proof of Theorem \ref{Theorem 2.1}, cf.\ \eqref{PP1}, we decompose the determinant of the matrix using the first row, and then use  formula \eqref{3.8}  to obtain the identity
\begin{equation}\label{Emat}
\det\big( -zD+i\sqrt{\lambda}sD_p^0+A\big)=\psi(z)+i\sqrt{\lambda}s\widehat{\psi}(z)
=s\big(\phi_N(\lambda)+i\sl\phi_D(\lambda)\big)=sE(\sl).
\end{equation}
Since the matrix $-z_0D+i\sqrt{\lambda}s_0D_p^0+A$ is self-adjoint, the geometric multiplicity of its eigenvalue $0\in\spec\big(-z_0D+i\sqrt{\lambda}s_0D_p^0+A\big)$ is equal to the  algebraic multiplicity which, in turn, is equal by \eqref{Emat} to the multiplicity of $\sqrt{\lambda_0}$ as a root of the function $E(\cdot)$. So, to complete the proof of the lemma it remains to show that $\ker(L_\infty-\lambda_0)$ and $\ker\big(-z_0D+i\sqrt{\lambda}s_0D_p^0+A\big)$ are isomorphic. To this end, as in the proof of Theorem \ref{Theorem 3.1prime}, we fix a $y=(y_j)_{j=0}^g\in\ker(L_\infty-\lambda_0)$. In particular, \eqref{*} holds for $y_0$ by Lemma \ref{phiphi} (iv). We define $\iota(y):=(Y_v)_{v\in\cV}\in\bbC^{p\times1}$ with entries $Y_v$ given in that proof, see the discussion around \eqref{defYv}. Since $y\in\dom(L_\infty)$, the boundary conditions in \eqref{4.3} hold; in particular, since
$y_0(0)=\phi_D(\lambda_0)$ and $y'_0(0)=-\phi_N(\lambda_0)$ by \eqref{*}, we conclude that $Y_{v_0}=s_0^{-1}y_k(0)=s_0^{-1}y_0(0)=s_0^{-1}\phi_D(\lambda_0)$ for $k=1,\ldots,d(v_0)$. Formulas \eqref{ye}, \eqref{yeder}, \eqref{yeder2} at $v\neq v_0$ continue to hold. Applying Kirchhoff condition at $v_0$  and the first formula in \eqref{yeder}, we infer
\begin{equation}\label{yeder3}\begin{split}
0&=y'_0(0)+\sum_{k=1}^{d(v_0)}y'_k(0)=-\phi_N(\lambda_0)+\sum_{k=1}^{d(v_0)}\big(Y_{v_k}-z_0Y_{v_0}\big)\\&=
-\phi_N(\lambda_0)-z_0d(v_0)Y_{v_0}+\sum_{k=1}^{d(v_0)}Y_{v_k}=
\big(-zd(v_0)+i\sqrt{\lambda_0}s_0\big)Y_{v_0}+\sum_{v'\in\cV}a_{v_0v'}Y_{v_0},
\end{split}
\end{equation}
where $a_{vv'}$ are the entries of $A$ and we have used the equality $-\phi_N(\lambda_0)=i\sqrt{\lambda_0}\phi_D(\lambda_0)=i\sqrt{\lambda_0}s_0Y_{v_0}$ that follows from $E(\sqrt{\lambda_0})=0$ and the definition of $Y_{v_0}$. Formulas \eqref{yeder2} and \eqref{yeder3} show $\iota(y)\in\ker\big(-z_0D+i\sqrt{\lambda}s_0D_p^0+A\big)$ and then $\iota$ is the required isomorphism by the same argument as in the proof of Theorem \ref{Theorem 3.1prime}. This concludes the first part of the proof of the theorem.

In the second part of the proof we impose Hypothesis \ref{hypo:stass} and make a homotopy argument similar to the second part of the proof of Theorem \ref{Theorem 3.1prime}. Indeed, let us consider the family of potentials $q(\cdot,\eta)$, $\eta\in[0,1]$, similarly to \eqref{qeta},
\begin{equation}\label{qeta2}
q(x,\eta)=\eta^2\widetilde{q}(\ell-x)+\eta\big(\widetilde{q}(x)-q_j(x)\big)+q_j(x),\, x\in e_j, j=1,\ldots,g, \eta\in[0,1],
\end{equation} 
and set $q(x,\eta)=0$ for $x\in e_0$.  Here, $\widetilde{q}\in L^2(0,\ell)$  is any real potential such that $\widetilde{q}(x)\le 0$ and $\widetilde{q}(x)-q_j(x)\le 0$ and  for almost all $x\in e_j$ and all $j=1,\ldots,g$.  
We denote by $L_\infty(\eta)$ the operator in $L^2(T_\infty)$  that corresponds to the potential $q(\cdot,\eta)$ so that $L_\infty(0)=L_\infty$, the original operator, and notice that $q(\cdot,1)$ satisfies Hypothesis \ref{hyp:two}. Thus, the multiplicities of the eigenvalues of $L_\infty(1)$ and the respective zeros of the Jost function $E(\sl,1)$ are equal by the first part of the proof.

Acting as in the proof of Theorem \ref{Theorem 3.1prime}, we use perturbation theory for the holomorphic family of self-adjoint operators $L_\infty(\eta)$ from \cite[Section VII.3]{Kato} to construct holomorphic curves of the eigenvalues $\lambda(\eta)\in\spec(L_\infty(\eta))$ that start at the negative eigenvalues $\lambda(0)\in\spec(L_\infty(0))$.
By the  perturbation theory the holomorphic curves $\lambda(\cdot)$ are such that $n(\lambda(0),0)=n(\lambda(1),1)$. Also, we use \cite[Theorem 9.1.1]{MP0} to show that $m(\lambda(0),0)=m(\lambda(1),1)$; 
here we denote by $n(\lambda,\eta)$ the (geometric) multiplicity of $\lambda\in\spec(L_\infty(\eta))$ and by $m(\sl,\eta)$ the multiplicity of $\sl$ as a zero of the function $\sl\mapsto E(\sl,\eta)=\phi_N(\lambda,\eta)+i\sl\phi_D(\lambda,\eta)$.

 We use a formula for the $\eta$-derivative of $\lambda(\cdot)$ sometimes called the  Hellmann-Feynman formula and sometimes Hadamard variational formula, cf., e.g.,
\cite[Theorem VII.3.6, (3.18)]{Kato},  or, for a more general situation, \cite[Theorem 3.25]{LatSuk}. It asserts that 
\begin{equation}\label{HF}
\frac{d\lambda(\eta)}{d\eta}=\big\langle\frac{d L_\infty(\eta)}{d\eta}u(\cdot, \eta),u(\cdot, \eta)\big\rangle_{L^2(T_\infty)}=\big\langle\frac{\partial q(\cdot,\eta)}{\partial\eta}u(\cdot, \eta),u(\cdot, \eta)\big\rangle_{L^2(T_\infty)},\, \eta\in[0,1],
\end{equation}
where $u(\cdot, \cdot)$ is the holomorphic in $\eta$ family of eigenfunctions so that $L_\infty(\eta)u(\cdot, \eta)=\lambda(\eta)u(\cdot, \eta)$ and $\|u(\cdot, \eta)\|_{L^2(T_\infty)}=1$. 
%
Now \eqref{qeta2} and the choice of $\widetilde{q}$ in \eqref{qeta2} show that \begin{equation*}
\frac{d\lambda(\eta)}{d\eta}=\sum_{j=1}^g\big\langle  \big(2\eta \widetilde{q}(\ell-\cdot)+\widetilde{q}(\cdot)-q_j(\cdot)\big)u(\cdot, \eta), u(\cdot, \eta)\big\rangle_{L^2(e_j)} \le 0\end{equation*}
and so $\eta\mapsto\lambda(\eta)$ is a nonincreasing function of $\eta\in[0,1]$. Thus, if $\eta$ changes from 1 to 0 then the eigenvalues $\lambda(\eta)$ move along the real line to the right. 
Since $q(\cdot,1)$ satisfies Hypothesis \ref{hyp:two}, by the first part of the proof of the theorem we have
$n(\lambda(1),1)=m(\sqrt{\lambda(1)},1)$. This yields $n(\lambda(0),0)=m(\sqrt{\lambda(0)},0)$ as required to finish the proof of the theorem.
\end{proof}


The next result relates the number of isolated eigenvalues of  $L_\infty$ and $L_N$. The proof is based on the homotopy techniques for the operator pencils from \cite[Chapter 1]{MP0}.

\begin{lemma}\label{lem:LinfLN} Under Hypothesis \ref{hypo:stass}, the number of the isolated (negative) eigenvalues of $L_\infty$ is equal to the number of the isolated (negative) eigenvalues of $L_N$ (in both cases counting their (geometric) multiplicities).
\end{lemma}
\begin{proof} 
To begin, we introduce two $(g+1)\times(g+1)$ matrices,  $M=\diag\{1,\ldots,1,0\}$ and $K=\diag\{0,\ldots,0,1\}$, and the self-adjoint operator $\wL$
acting in the space $L^2(T)\oplus\bbC=\big(\oplus_{j=1}^gL^2(0,\ell)\big)\oplus\bbC$ by the rule \begin{equation}\label{defwL}
\wL:Y\mapsto\begin{bmatrix}(-y''_j+q_jy_j)_{j=1}^g\\
\sum_{k=1}^{d(v_0)}y'_k(0)\end{bmatrix}; \text{ here and below we denote } Y=\begin{bmatrix}(y_j)_{j=1}^g\\c\end{bmatrix}.\end{equation}
We define the domain $\dom(\wL)$ as the set of $Y$'s such that $y_j\in H^2(0,\ell)$, $j=1,\ldots,g$, and $c\in\bbC$ satisfy the standard boundary conditions at all vertices except the root $v_0$ and the conditions $y_k(0)=c$, $k=1,\ldots,d(v_0)$, at the root.
Also, we introduce the operator pencils
\begin{equation}\label{defPK}
P_K(\lambda):=\lambda M+i\sl K-\wL \text{ and } P_0(\lambda):=\lambda M-\wL \text{ for $\lambda\in\bbC$}
\end{equation}
with the domains $\dom(\wL)$. The pencil $P_K$ is the quadratic in $\sl$ pencil considered in \cite[Chapter 1]{MP0} with $\lambda$ replaced by $-\sl$ to account for the plus sign in \eqref{defPK} in front of $i\sl K$ versus the minus sign used in \cite{MP0}. It is easy to see that \cite[Condition 1]{MP0} is satisfied for $P_K$ and thus \cite[Lemma 1.2.1]{MP0} and \cite[Theorem 1.3.3]{MP0} apply. In particular, we conclude from the two results that 
\begin{equation}\label{spsp}
\card\{\lambda\in\spec(P_K): \im(\sl)>0, \re(\sl)=0\}=\card\{\lambda\in\spec(\wL):\lambda<0\},
\end{equation}
where we count the eigenvalues with their multiplicities and recall that $\lambda\in\spec(P)$ for a pencil $P$ if and only if there is a nonzero $Y\in\ker(P(\lambda))$. We claim that 
\begin{equation}\label{PKP0}
\card\{\lambda\in\spec(\wL):\lambda<0\}=\card\{\lambda\in\spec(P_0): \im(\sl)>0, \re(\sl)=0\}.
\end{equation}
Postponing the proof of the claim, we use \eqref{PKP0} to derive
 the required equality
 \begin{equation}\label{LinfLN}
 \card\{\lambda\in\spec(L_\infty):\lambda<0\}=\card\{\lambda\in\spec(L_N):\lambda<0\}
 \end{equation}
 as follows. Assume that $\lambda<0$ and $\im(\sl)>0$. The assertion $Y\in\ker(P_0(\lambda))\subset\dom(\wL)$ is equivalent to the fact that the functions $y_j$, $j=1,\ldots,g$, satisfy the eigenvalue equations \eqref{4.1}, the standard boundary conditions at all vertices but $v_0$, the continuity condition $y_k(0)=c$, $k=1,\ldots,d(v_0)$, and the Kirchhoff's(-Neumann) condition $\sum_{k=1}^{d(v_0)}y'_k(0)=0$ at $v_0$. In other words, $(y_j)_{j=1}^g\in\ker(L_N-\lambda)$ if and only is $Y\in\ker(P_0(\lambda))$. Therefore
  \begin{equation}\label{LinfLN2}
\card\{\lambda\in\spec(P_0): \im(\sl)>0, \re(\sl)=0\}=\card\{\lambda\in\spec(L_N):\lambda<0\}.
 \end{equation}
Furthermore, the assertion $Y\in\ker(P_K(\lambda))\subset\dom(\wL)$ is equivalent to the fact that the functions $y_j$, $j=1,\ldots,g$, satisfy the eigenvalue equations \eqref{4.1}, the standard boundary conditions at all vertices but $v_0$, the continuity condition $y_k(0)=c$, $k=1,\ldots,d(v_0)$, and the Kirchhoff's(-Robbin) condition $i\sl c=\sum_{k=1}^{d(v_0)}y'_k(0)$ at $v_0$, where the last equation follows because the last entry of $P_K(\lambda)Y$ is zero. For the given $y_j$ on $e_j$, $j=1,\ldots,g$, we now define the function $y_\infty$ on $T_\infty$ by letting $y_\infty(x)=y_j(x)$ for $x\in e_j$ and $j=1,\ldots,g$, and $y_\infty(x)=y_1(0)e^{i\sl x}$ for $x\in\e_0=[0,\infty)$. Due to $\im(\sl)>0$, the function $y_\infty$ exponentially decays on $e_0$, and, conversely, any $L^2$-solution to \eqref{4.2} on $e_0$ is of the form $ce^{i\sl x}$. We conclude that $y_\infty\in
\ker(L_\infty-\lambda)$ if and only is $Y\in\ker(P_K(\lambda)$ with $c=y_1(0)$, and therefore
  \begin{equation}\label{LinfLN3}
\card\{\lambda\in\spec(P_K): \im(\sl)>0, \re(\sl)=0\}=\card\{\lambda\in\spec(L_\infty):\lambda<0\}.
 \end{equation}
 Clearly, \eqref{LinfLN} follows from \eqref{spsp}, \eqref{PKP0}, \eqref{LinfLN2}, \eqref{LinfLN3}.
 
 It remains to justify the claim \eqref{PKP0}. The proof below is a version of the homotopy argument frequently used in \cite[Section 1.3]{MP0}. Indeed, we introduce a family of linear operator pencils,
 \begin{equation}\label{defPeta}
 P(\lambda,\eta):=(I-\eta K)\lambda-\wL, \text{ where $\eta\in[0,1]$},
 \end{equation}
 so that $P(\lambda,0)=\lambda-\wL$ and $P(\lambda,1)=\lambda M-\wL=P_0(\lambda)$. We consider a (differentiable) curve of the eigenvalues $\lambda=\lambda(\eta)$ and the respective eigenvectors $Y=Y(\eta)\in\ker(P(\lambda,\eta))$ of unit norm in $L^2(T)\oplus\bbC$. We differentiate in $\eta$ the eigenvalue equation $P(\lambda(\eta),\eta)Y(\eta)=0$, compute the scalar product in $L^2(T)\oplus\bbC$ of the result with $Y(\eta)$, use that $\wL$ and $K$ are self-adjoint, and then use the eigenvalue equation again. This results in the formula
 \begin{equation}\label{derlamb}
 \frac{d\lambda(\eta)}{d\eta}=f(\eta)\lambda(\eta), \, \eta\in[0,1],
 \text{ where $f(\eta):=(c(\eta))^2/\big(1-\eta(c(\eta))^2\big)$,}
 \end{equation}
 and $c(\eta)$ is the last entry of the vector $Y(\eta)$ as in \eqref{defwL} and $\|Y\|^2=\sum\|y_j\|_{L^2(0,\ell)}^2+c^2=1$.
 Solving the ODE \eqref{derlamb} shows that if $\lambda(0)<0$ then the function $\lambda(\cdot)$ is decreasing while if $\lambda(0)>0$ then the function $\lambda(\cdot)$ is increasing for $\eta\in[0,1]$. Thus, the number of the negative eigenvalues of $P(\lambda,0)$ and $P(\lambda,1)$ must be the same which is the required claim \eqref{PKP0}.  \end{proof}
 
In the next theorem, which immediately  follows from  Lemmas \ref{propE} and \ref{lem:LinfLN} and Theorems \ref{Theorem 3.1prime} and \ref{lem:EvsL}, we summarize the spectral picture for the isolated eigenvalues of $L_N$ and $L_\infty$ in terms of the characteristic and Jost functions.

\begin{theorem}\label{Theorem 4.2} 
We assume Hypothesis \ref{hypo:stass}, impose standard boundary conditions at all vertices of $T_\infty$, and let $q=0$ on $e_0$. Then:

\textbf{1.}\,
All zeros of the function $\sl\mapsto E(\sqrt{\lambda})$ are located either in the closed lower half-plane or on a finite interval of the positive imaginary half-axis. 

\textbf{2.}\, The number of zeros of the function $\sl\mapsto E(\sqrt{\lambda})$  located on the positive imaginary half-axis  (counted with their multiplicities) is the same as the number of negative  zeros of the function $\lambda\mapsto\phi_N(\lambda)$ (counted with their multiplicities) is the same as the number of the negative eigenvalues of the operator $L_N$ (counting their multiplicities) and is the same as the number of the negative eigenvalues of the operator $L_\infty$ (counting their multiplicities).
\end{theorem}

To conclude, we remark that the proof of Lemma \ref{lem:LinfLN} (and therefore of Theorem \ref{Theorem 4.2}) that we presented above is based on
the results in \cite[Chapter 1]{MP0} regarding quadratic operator pencils. There is yet another strategy of the proof based on the resuts in \cite[Section 5.2]{MP0} regarding the shifted Hermite-Biehler functions,
cf.\ \cite[Chapter 6]{MP0} and also  \cite{PUMZ07}. The second strategy, which will be explained elsewhere, does not involve operator pencils directly, and uses general functional properties of the Jost and characteristic functions.


\section{Recovering the shape of a tree by scattering data}\label{sec:5}

In this section we consider the scattering problem on the tree $T_\infty$ obtained by attaching a lead (half-infinite edge)  to one of the vertices of our equilateral  compact tree $T$ with $p=p_{{}_T}$ vertices. 

Our objective is to recover the shape of the tree $T$ provided we are given some scattering information obtained from the respective $S$-function \eqref{dfS} (the scattering function, cf.\ \cite{Taylor}). In fact, we show how to recover the polynomials $\psi(z)$ and $\widehat{\psi}(z)$ defined in \eqref{2.2} from given $2p+1$ scalars $f_k,\hat{f}_{\hat{k}}$, $k=0,1,\ldots,p$, $\hat{k}=0,1,\ldots,p-1$, which are the limiting values (along some  appropriately chosen sequences) of certain functions, $F$ and $\widehat{F}$, constructed by means of: (a) the $S$-function introduced in Section \ref{sec:scattering}, and (b) the common eigenvalues of the Dirichlet and Neumann eigenvalue problems on $T$. As soon as the ratio of the polynomials $\psi(z)$ and $\widehat{\psi}(z)$ is obtained, we use the algorithm described in Section \ref{sec:2} to recover the shape of the graph.

Given a graph $T$ one can attach the lead to its different vertices. The characteristic function $\phi_N$ will not change if we choose $v_0$ differently but $\phi_D$ might. However, there are graphs for which $\phi_D$ does not change, cf.\ \cite{Pistol21},
 thus providing  examples of  graphs with the same $S$-function but of different shape.

As in Section \ref{sec:3} and as in \cite{MuP}, we assume that the potential $q_0$ is identically zero on the lead $e_0$.
As in \eqref{dfS}--\eqref{4.5}  we  can write the $S$-function of the scattering problem on $T_\infty$ as 
\begin{equation}
\label{5.1}
S(\sl)=\frac{\phi_N(\lambda)-i\sqrt{\lambda}\phi_D(\lambda)}{\phi_N(\lambda)+i\sqrt{\lambda}\phi_D(\lambda)}=\frac{\widetilde{E}(-\sl)}{\widetilde{E}(\sl)},\, \sl\in\bbC,
\end{equation}
where $\widetilde{E}(\sl)$ and $\widetilde{E}(-\sl)$ have no common zeros.

Unlike the standard scattering theory with no graph $T$ attached to the lead,
there may exist common zeros of $\phi_N(\lambda)$ and $\phi_D(\lambda)$. The common zeros (in other words, the points in $\spec(L_D)\cap\spec(L_N)$, cf.\ also Lemma \ref{rem:intersection}) may be either negative or nonnegative. In the former case the respective values of $\lambda$ correspond to the pure imaginary $\sl$ and give isolated eigenvalues of $L_\infty$ while in the latter case the respective values of $\lambda$ are the eigenvalues of the operator $L_\infty$ in $L^2(T_\infty)$ embedded into its essential spectrum, cf.\ Lemma \ref{propE}.
In any case, the corresponding to the common zeros of $\phi_D$ and $\phi_N$ factors in the numerator and in the denominator of the first fraction in \eqref{5.1} cancel each other. This tells us that to be able to recover the polynomials $\psi$ and $\widehat{\psi}$ we will need to know not only the values of the $S$-function but the location of the common zeros of $\phi_D$ and $\phi_N$. 

To construct $\widetilde{E}$ in the second fraction in \eqref{5.1}, we recall from \cite{Levin} 
that any entire function of exponential type, say, $\phi_N$ or $\phi_D$, could be written as the product of the following three expressions:  (a) a (convergent when infinite) product of the terms $(1-\frac{\lambda}{\lambda_j})$ where $\lambda_j\neq0$ are the zeros of the function;  (b) the term $\lambda^{m_0}$ where $m_0$ is the multiplicity of $\lambda=0$ as a zero of the function;   (c) a nonzero constant. 

Let us denote by $\lambda_j^0$, $j=1,\ldots,j_0$, $1\le j_0\le+\infty$, all common zeros $\lambda_j^0\neq0$ of $\phi_D$ and $\phi_N$ if there are any, and set $j_0=0$ if $\phi_D$ and $\phi_N$ have no common zeros. 
Also, we denote by $m$ the smallest of the following two numbers: the multiplicity of $\lambda=0$ as a zero of the function $\phi_D$ and the multiplicity of $\lambda=0$ as a zero of the function $\phi_N$; we set $m=0$ if at least one of the numbers $\phi_D(0)$ or $\phi_N(0)$ is not equal to zero.
We  then write
\begin{equation}\label{eq1.15}
\phi_D(\lambda)=\lambda^m\prod_{j=1}^{j_0}(1-\frac{\lambda}{\lambda_j^0})\widetilde{\phi}_D(\lambda) \, \text{ and } \,
\phi_N(\lambda)=\lambda^m\prod_{j=1}^{j_0}(1-\frac{\lambda}{\lambda_j^0})\widetilde{\phi}_N(\lambda),
\end{equation}
where $\widetilde{\phi}_D$ and $\widetilde{\phi}_N$ have no common zeros. Here and in what follows we set $\prod_{j=1}^0=1$. Letting $\widetilde{E}(\sl)=\widetilde{\phi}_N(\lambda)+i\sl\,\widetilde{\phi}_D(\lambda)$ gives the second  fraction in \eqref{5.1}. Equation \eqref{5.1} can be also re-written as follows,
\begin{equation*}\begin{split}
\phi_N(\lambda)+i\sl\phi_D(\lambda)&=
\lambda^m\prod_{j=1}^{j_0}(1-\frac{\lambda}{\lambda_j^0})\widetilde{E}(\sl),\\
\phi_N(\lambda)-i\sl\phi_D(\lambda)&=
\lambda^m\prod_{j=1}^{j_0}(1-\frac{\lambda}{\lambda_j^0})\widetilde{E}(-\sl).
\end{split}
\end{equation*}
From these equations we obtain formulas expressing $\phi_N(\lambda)$ and  $\phi_D(\lambda)$ via $\widetilde{E}(\sl)$ and $\widetilde{E}(-\sl)$.  Next, expressing
$\psi(\cos(\sl\ell))$ and $\widehat\psi(\cos(\sl\ell))$ from equations \eqref{nass}, \eqref{dass} in Lemma \ref{lem:phipsi} via $\phi_N(\lambda)$ and  $\phi_D(\lambda)$ just obtained  yields the following asymptotic relations,
\begin{equation}\label{psipsihat}\begin{split}
\psi(\cos(\sl\ell))&=F(\sl)+O(1/\sl),\,\lambda>0,\lambda\to+\infty,\\ \widehat{\psi}(\cos(\sl\ell))&=\widehat{F}(\sl)+O(1/\sl),\,\lambda>0,\lambda\to+\infty,\end{split}
\end{equation}
where we introduced notation
\begin{equation}\label{defF}\begin{split}
F(\sl) &= \frac{\sin(\sl\ell)}{2\sl} \lambda^m\prod_{j=1}^{j_0}(1-\frac{\lambda}{\lambda_j^0})
\big(\widetilde{E}(\sl)+\widetilde{E}(-\sl)\big),\\
\widehat{F}(\sl)&=\frac{1}{2i\sl}\lambda^m\prod_{j=1}^{j_0}(1-\frac{\lambda}{\lambda_j^0})
\big(\widetilde{E}(\sl)-\widetilde{E}(-\sl)\big).\end{split}
\end{equation}
To describe how to recover from the scattering information the polynomials $\psi$ and $\widehat{\psi}$ defined in \eqref{2.2}  we need one last piece of notation: We introduce the following $2p+1$ numbers,  $z_k,\hat{z}_{\hat{k}}\in [0,1]$, and the following real sequences, $\sqrt{\lambda_k^{(n)}}$, $\sqrt{\hat{\lambda}_k^{(n)}}$,
\begin{equation}\label{zlambda}\begin{split}
z_k&:=\frac{k}{p},\, \hat{z}_{\hat{k}}:=\frac{\hat{k}}{p-1},\, k=0,1,\ldots, p,\, \hat{k}=0,1,\ldots,p-1,\\
\sqrt{\lambda_k^{(n)}}&:=(\arccos z_k+2\pi n)/\ell,
\sqrt{\hat{\lambda}_{\hat{k}}^{(n)}}:=(\arccos\hat{z}_{\hat{k}}+2\pi n)/\ell,\,
\, n=1,2,\ldots.\end{split}
\end{equation}
We stress that the function $F$ and $\widehat{F}$ from \eqref{defF} are defined using only the $S$-function of the graph $T_\infty$ and the common zeros of $\phi_D$ and $\phi_N$. It is therefore natural to call the following $2p+1$ numbers, $f_k$ and $\hat{f}_{\hat{k}}$, {\em the scattering information for $T_\infty$},
\begin{equation}\label{defSD}
f_k:=\lim_{n\to\infty}F\big(\sqrt{\lambda_k^{(n)}}\big),\, k=0,1,\ldots,p,\,
\hat{f}_{\hat{k}}:=\lim_{n\to\infty}\widehat{F}\big(\sqrt{\hat{\lambda}_{\hat{k}}^{(n)}}\big),\, \hat{k}=0,1,\ldots, p-1.
\end{equation}
We recall that in Section \ref{sec:2} we described an algorithm of determining the shape of a tree when the ratio of the respective polynomials $\psi$ and $\widehat{\psi}$ is given. The following result shows how to recover the polynomials provided the scattering information $f_k$, $\hat{f}_{\hat{k}}$ is given.


\begin{theorem}\label{thm:sec5}  Let $T_\infty$ be the graph obtained from a compact tree $T$ by attaching an infinite lead $e_0$ at its root. We assume  Hypothesis \ref{hypo:stass} and $q_0=0$ and impose standard boundary conditions. Given the scattering information $f_k$ and $\hat{f}_{\hat{k}}$ defined in \eqref{defF}, \eqref{defSD} via the $S$-function of $T_\infty$ and the set $\spec(L_D)\cap\spec(L_N)$ of the common eigenvalues of $L_D$ and $L_N$, the polynomials $\psi=\psi(z)$ and $\widehat{\psi}=\widehat{\psi}(z)$ associated with $T$ as indicated in \eqref{2.2} are uniquely determined by finding their coefficients from the relations $\psi(z_k)=f_k$, $k=0,1,\ldots,p$, and $\widehat{\psi}(\hat{z}_{\hat{k}})=\hat{f}_{\hat{k}}$, $\hat{k}=0,1,\ldots,p-1$.
\end{theorem}
\begin{proof}
Due to \eqref{reprform} 
 for real $\lambda$ we have
\begin{equation}\label{assimppsi}
{\psi}(c(\lambda,\ell))-{\psi}(\cos\sqrt{\lambda} \ell)=o(1),\,
\widehat{\psi}(c(\lambda,\ell))-\widehat{\psi}(\cos\sqrt{\lambda} \ell)=o(1) \text{ as $\lambda\to+\infty$.}
\end{equation}
We continue the proof for $\psi$, the proof for $\widehat{\psi}$ is analogous. Formulas \eqref{psipsihat}, \eqref{zlambda}, \eqref{defSD}, and \eqref{assimppsi} show that
\begin{equation}
f_k=
\lim_{n\to\infty}F\big(\sqrt{\lambda_k^{(n)}}\big)
=\lim_{n\to\infty}\psi\big(c(\sqrt{\lambda_k^{(n)}}, \ell)\big)
=\lim_{n\to\infty}\psi\big(\cos(\sqrt{\lambda_k^{(n)}}\ell)\big)=\psi(z_k),
\end{equation}
and thus $f_k$ determine the values $\psi(z_k)$ of the polynomial $\psi$ of degree $p$ at $p+1$ different points $z_k$, $k=0,1,\ldots,p$. Denoting by $\psi_i$ the coefficients of the polynomial $\psi(z)=\sum_{i=0}^p\psi_iz^i$ we conclude that the relations $\psi(z_k)=f_k$, $k=0,1,\ldots,p$, give a system of $p+1$ linear equations with $p+1$ unknowns $\psi_i$. The matrix of the system has a nonzero Vandermonde determinant $\det[(z_k)^i]_{k,i=0}^p$ and thus the system has a unique solution.
\end{proof}

\appendix
\section{Tables for $\Phi_D$ and $\Phi_N$}\label{TT}
   \[\begin{array}{c|cccccccccccccc}
 \Phi_D & 1 &2 \dots & d_0 & {11}  \dots & 1(d_1-1) & \dots & 1  \dots & d_0 & {11}  \dots & 1(d_1-1) \dots \\ \hline
   1 & 1 &0 \dots & 0 &0 \dots & 0 & \dots & 0  \dots & 0 & 0  \dots & 0 \dots \\
    2 & 0 &1 \dots & 0 &0 \dots & 0 & \dots & 0  \dots & 0 & 0  \dots & 0 \dots \\
 \dots &  \dots &\dots & \dots & \dots &  \dots & \dots & \dots & \dots & \dots & \dots \\
 d_0 & 0 & 0 \dots & 1 &0 \dots & 0 & \dots & 0  \dots & 0 & 0  \dots & 0 \dots \\
 1 & c(\ell) &0  \dots & 0 & -1 \dots &0 & \dots & s(\ell) \dots & 0   & 0 \dots & 0  \dots \\
  \dots & \dots & \dots & \dots & \dots &  \dots & \dots & \dots & \dots & \dots & \dots \\
   d_1-1 & c(\ell) &0 \dots & 0 & 0 \dots &-1 & \dots & s(\ell) \dots & 0   & 0 \dots & 0  \dots \\
    d_1 & c'(\ell) &0  \dots & 0 & 0 \dots &0 & \dots & s'(\ell) \dots & 0   & -1\dots & -1  \dots \\
     \dots & \dots & \dots & \dots &  \dots & \dots & \dots & \dots & \dots & \dots \\
     1 & 0 &0 \dots & c(\ell) & 0 \dots &0 & \dots & 0 \dots & s(\ell)   & 0 \dots & 0  \dots \\
  \dots & \dots & \dots & \dots & \dots &  \dots & \dots & \dots & \dots & \dots & \dots \\
   d_{d_0}-1 & 0 &0 \dots & c(\ell) & 0 \dots &0 & \dots & 0 \dots & s(\ell)   & 0 \dots & 0  \dots \\
    d_{d_0} & 0 & \dots & c'(\ell) & 0 \dots &0 & \dots & 0 \dots & s'(\ell)   & 0\dots & 0  \dots \\
  \vdots & \vdots & \vdots & \vdots & \vdots & \vdots & \vdots & \vdots & \vdots  & \vdots &   \vdots
\end{array}.\]

\[\begin{array}{c|cccccccccccccc}
 \Phi_N & 1 &2  \dots & d_0 & {11}  \dots & 1(d_1-1) & \dots & 1  \dots & d_0 & {11}  \dots & 1(d_1-1) \dots \\ \hline
   1 & 1 &-1 \dots & 0 &0 \dots & 0 & \dots & 0  \dots & 0 & 0  \dots & 0 \dots \\
    2 & 1 &0 \dots & 0 &0 \dots & 0 & \dots & 0  \dots & 0 & 0  \dots & 0 \dots \\
 \dots &  \dots &\dots & \dots & \dots &  \dots & \dots & \dots & \dots & \dots & \dots \\
 d_0-1 & 1 & 0 \dots & -1 &0 \dots & 0 & \dots & 0  \dots & 0 & 0  \dots & 0 \dots \\
 d_0 & 0 & 0 \dots & 0 &0 \dots & 0 & \dots & 1 \dots & 1 & 0  \dots & 0 \dots \\
 1 & c(\ell) &0  \dots & 0 & -1 \dots &0 & \dots & s(\ell) \dots & 0   & 0 \dots & 0  \dots \\
  \dots & \dots & \dots & \dots & \dots &  \dots & \dots & \dots & \dots & \dots & \dots \\
   d_1-1 & c(\ell) &0 \dots & 0 & 0 \dots &-1 & \dots & s(\ell) \dots & 0   & 0 \dots & 0  \dots \\
    d_1 & c'(\ell) &0  \dots & 0 & 0 \dots &0 & \dots & s'(\ell) \dots & 0   & -1\dots & -1  \dots \\
     \dots & \dots & \dots & \dots &  \dots & \dots & \dots & \dots & \dots & \dots \\
     1 & 0 &0 \dots & c(\ell) & 0 \dots &0 & \dots & 0 \dots & s(\ell)   & 0 \dots & 0  \dots \\
  \dots & \dots & \dots & \dots & \dots &  \dots & \dots & \dots & \dots & \dots & \dots \\
   d_{d_0}-1 & 0 &0 \dots & c(\ell) & 0 \dots &0 & \dots & 0\dots & s(\ell)    & 0 \dots & 0  \dots \\
    d_{d_0} & 0 & \dots & c'(\ell) & 0 \dots &0 & \dots & 0 \dots & s'(\ell)   & 0\dots & 0  \dots \\
  \vdots & \vdots & \vdots & \vdots & \vdots & \vdots & \vdots & \vdots & \vdots  & \vdots &   \vdots
\end{array}.\]


 
  
 \subsection*{Acknowledgements}

This material is based upon work supported by the US NSF grants DMS-2108983/2106157; the authors thank the US NSF, National Academy of Sciences and Office of Naval Research for the support  of the project  ``IMPRESS-U: Spectral and geometric methods for damped wave equations with applications to fiber lasers".
The hospitality of the Institute of Mathematics of the Polish Academy of Sciences where the work began is gratefully acknowledged. VP was partially supported by the Academy of Finland (project no. 358155) and  is grateful to the University of Vaasa for hospitality. OB and VP thank the Ministry of Education and Science of Ukraine for the support in completing the work on the project 'Inverse problems of finding the shape of a graph by spectral data', state registration number 0124U000818.  YL would
like to thank the Courant Institute of Mathematical Sciences at NYU
and especially Prof. Lai-Sang Young for their hospitality.

The authors thank Prof.\ Uzy Smilansky for suggesting relevant literature.

\end{document}